\documentclass[VANCOUVER,STIX1COL]{WileyNJD-v2}

\articletype{Article Type}%

\received{26 April 2016}
\revised{6 June 2016}
\accepted{6 June 2016}

\usepackage{pgfplots} 
\pgfplotsset{compat=1.13} 
\usetikzlibrary{arrows.meta} 
\usepackage{bm} 
\usepackage{multirow} 
\usepackage{makecell} 
\usepackage{pifont}      
\usepackage{boldline}  

\usepackage{xcolor}

\newcommand{\xmark}{\ding{55}}
\renewcommand{\d}[0]{\ensuremath{\operatorname{d}\!}}
\DeclareMathOperator*{\argmin}{arg\,min}
\DeclareMathOperator*{\diag}{diag}

\begin{document}

\title{Optimizing multigrid reduction-in-time (MGRIT) and Parareal coarse-grid operators for linear advection\protect\thanks{A published version of this article may be found at \url{https://doi.org/10.1002/nla.2367}.}}

\author[1]{Hans De Sterck}

\author[2]{Robert D. Falgout}

\author[3]{Stephanie Friedhoff}

\author[4]{Oliver A. Krzysik*}

\author[5]{Scott P. MacLachlan}

\authormark{DE STERCK \textsc{et al}.}

\address[1]{\orgdiv{Department of Applied Mathematics}, \orgname{University of Waterloo}, \orgaddress{Waterloo, Ontario, \country{Canada}}}

\address[2]{\orgdiv{Center for Applied Scientific Computing}, \orgname{Lawrence Livermore National Laboratory}, \orgaddress{Livermore, California, \country{USA}}}

\address[3]{\orgdiv{Department of Mathematics}, \orgname{Bergische Universit\"at Wuppertal}, \orgaddress{Wuppertal, \country{Germany}}}

\address[4]{\orgdiv{School of Mathematics}, \orgname{Monash University}, \orgaddress{Clayton, Victoria, \country{Australia}}}

\address[5]{\orgdiv{Department of Mathematics and Statistics}, \orgname{Memorial University of Newfoundland}, \orgaddress{St John's, Newfoundland and Labrador, \country{Canada}}}

\corres{*\email{oliver.krzysik@monash.edu}}

\abstract[Summary]{
Parallel-in-time methods, such as multigrid reduction-in-time (MGRIT) and Parareal, provide an attractive option for increasing concurrency when simulating time-dependent PDEs in modern high-performance computing environments. While these techniques have been very successful for parabolic equations, it has often been observed that their performance suffers dramatically when applied to advection-dominated problems or purely hyperbolic PDEs using standard rediscretization approaches on coarse grids. In this paper, we apply MGRIT or Parareal to the constant-coefficient linear advection equation, appealing to existing convergence theory to provide insight into the typically non-scalable or even divergent behavior of these solvers for this problem. To overcome these failings, we replace rediscretization on coarse grids with improved coarse-grid operators that are computed by applying optimization techniques to approximately minimize error estimates from the convergence theory. One of our main findings is that, in order to obtain fast convergence as for parabolic problems, coarse-grid operators should take into account the behavior of the hyperbolic problem by tracking the characteristic curves. Our approach is tested for schemes of various orders using explicit or implicit Runge-Kutta methods combined with upwind-finite-difference spatial discretizations. In all cases, we obtain scalable convergence in just a handful of iterations, with parallel tests also showing significant speed-ups over sequential time-stepping.
}



\maketitle

\section{Introduction}
\label{sec:intro}
As compute clock speeds stagnate and core counts of parallel machines increase dramatically, parallel-in-time methods provide an attractive option for increasing concurrency when simulating time-dependent partial differential equations (PDEs).\cite{Gander2015, Ong_Schroder_2020} Two well-known parallel time integration methods are Parareal \cite{Lions_etal_2001} and multigrid reduction-in-time (MGRIT),\cite{Falgout_etal_2014} which can be considered as multigrid methods, although Parareal also has other interpretations.\cite{Gander_Vandewalle2007} For a wide variety of diffusion-dominated problems, these algorithms can achieve a significant reduction in wall clock time over sequential time-stepping methods, given enough parallel resources.\cite{Lions_etal_2001, Bal_Maday_2002, Falgout_etal_2014, Falgout_etal_2017, Falgout_etal_2019, Oneill_etal_2017, Ong_Schroder_2020} 
However, despite this success for diffusion-dominated problems, MGRIT and Parareal (along with most other parallel-in-time methods) tend to perform quite poorly on hyperbolic PDEs. Specifically, they typically exhibit extremely slow convergence or even divergence when applied to advection-dominated PDEs.\cite{Chen_etal_2014, Dai_Maday2013, Dobrev_etal_2017, Gander_Vandewalle2007, Gander2008, Hessenthaler_etal_2018, Howse_etal_2019, Howse2017_thesis, Nielsen_etal_2018, Ruprecht_Krause2012, Ruprecht2018, Schmitt_etal_2018, Steiner_etal_2015} 
Moreover, many of these examples demonstrate a clear deterioration in convergence as the amount of dissipation in the underlying PDE and/or its discretization is decreased.\cite{Howse2017_thesis, Schmitt_etal_2018, Steiner_etal_2015} 
Note that both MGRIT and Parareal converge to the exact solution of the discrete problem after a finite number of iterations due to sequential propagation of the initial condition across the temporal domain by the relaxation scheme.\cite{Falgout_etal_2014} As such, in this paper, we use the term `divergence' to describe solvers that converge to the solution only in a number of iterations close to that for which they would reach the exact solution.

To date, attempts to address this problem have been largely unsuccessful.
Several so-called `stabilized' variants of Parareal have been developed which overcome the typically divergent behavior of standard Parareal, but are significantly more expensive, and so their practicability is limited.\cite{Chen_etal_2014, Dai_Maday2013, Ruprecht_Krause2012}
Semi-Lagrangian coarse-grid operators were considered by Schmitt et al.\cite{Schmitt_etal_2018} for the viscous Burgers equation, inspiring some of the work in this paper. Unfortunately, though, their numerical tests showed that convergence deteriorates significantly in the zero-viscosity/hyperbolic limit.\cite{Schmitt_etal_2018}
Coarsening also in space was considered by Howse et al.\cite{Howse_etal_2019} when applying MGRIT to the linear advection and Burgers' equations to provide both cheaper multigrid cycles and overcome stability issues arising from coarsening only in time. Parallel speed-up was demonstrated for linear advection only, but convergence was, ultimately, slow and not scalable, and the approach did not work when applied to higher than first-order discretizations. 
Nielsen et al.\cite{Nielsen_etal_2018} achieved modest speed-ups for the shallow water equations, in part, by reducing the order of the discretizations in time and space, rather than coarsening the mesh.
Small to moderate speed-ups for linear hyperbolic PDEs were obtained with ParaExp,\cite{Gander_Guttel_2013} but this algorithm falls outside of the class of MGRIT-like algorithms that are the subject of this paper.
In all of these works, speed-ups over sequential time-stepping for hyperbolic PDEs are typically quite small (on the order of two to six), with slow convergence of the iteration ultimately inhibiting faster runtimes due to increased parallelism. For comparison, a speed-up on the order of 20 times was achieved for a diffusion-dominated parabolic problem in Falgout et al.\cite{Oneill_etal_2017} 

A number of theoretical convergence analyses have been developed for Parareal and MGRIT,\cite{DeSterck_etal_2019, Dobrev_etal_2017, Gander_Vandewalle2007, Gander2008, Hessenthaler_etal_2020, Ruprecht2018, Southworth2019, Gander_etal_2018, Friedhoff_MacLachlan2015} which have helped to explain numerical convergence results, and will likely play an important role in the design of new solvers. Furthermore, some theoretical studies have identified potential roadblocks for fast parallel-in-time convergence of hyperbolic PDEs.\cite{Gander_Vandewalle2007} Nevertheless, there does not yet exist a general understanding of why  the parallel-in-time solution of advection-dominated problems seems to be so much more difficult than for their diffusion-dominated counterparts.

The aim of this paper is to demonstrate that, in fact, MGRIT and Parareal, with the right choice of coarse-grid operator, can efficiently integrate hyperbolic PDEs. To do so, we work in an idealized environment, considering the constant-coefficient linear advection problem in one spatial dimension subject to periodic spatial boundary conditions, for which we can appeal to existing sharp MGRIT convergence theory. Informed by convergence theory and the PDE, we develop heuristics that coarse-grid operators should satisfy and we formulate optimization problems based on these to find `near-optimal' coarse-grid operators. For example, one such heuristic that we have developed is that coarse-grid operators should track information along characteristics, similar to the semi-Lagrangian schemes considered by Schmitt et al.\cite{Schmitt_etal_2018} However, our optimization-based coarse-grid operators lead to robust solvers in the hyperbolic limit, unlike their semi-Lagrangian coarse-grid operators.\cite{Schmitt_etal_2018}
We demonstrate that these coarse-grid operators lead to scalable convergence, in the sense that the computational work is almost linear asymptotically as a function of the problem size (i.e., the solvers converge in approximately a constant number of multigrid iterations asymptotically). Moreover, convergence is achieved in just a handful of iterations, for both implicit and explicit discretizations, resulting in significant speed-ups in parallel over sequential time-stepping, comparable to what has been achieved for parabolic PDEs. 
Notably, our results include the use of high-order accurate discretizations (up to fifth order), which is important because many results reported in the literature for hyperbolic PDEs have used diffusive, low-order discretizations that have likely aided the convergence of the given parallel-in-time method. Additionally, our approach works for large coarsening factors, and we employ fine-grid CFL numbers that reflect what would realistically be used with sequential time-stepping.

The optimization approaches presented in this paper rely on calculations that are feasible specifically for the case of one-dimensional, constant-coefficient linear advection. This precludes direct application of these approaches to more complicated hyperbolic PDEs. Crucially, though, they give us powerful tools to demonstrate that, for this canonical hyperbolic PDE, it is possible to obtain highly efficient MGRIT and Parareal solvers. 
At the same time, our main finding that coarse-grid operators should track characteristics is general, and we believe it will prove relevant to the design of solvers for more complicated hyperbolic PDEs. Practical methods for selecting coarse-grid operators for linear advection and other hyperbolic problems that follow this main insight are the subject of further research.

The remainder of this paper is organized as follows. In \S \ref{sec:prelims}, the model problem and its discretizations are introduced, a brief overview of MGRIT and Parareal is given, and some motivating numerical examples are presented. A discussion on convergence theory and what it reveals about the difficulty of hyperbolic problems is given in \S \ref{sec:convergence}. Section~\ref{sec:Psi_linear_approx} develops an optimization-based approach for finding effective coarse-grid time-stepping operators. Parallel results are given in \S \ref{sec:parallel} for some of the newly developed coarse-grid operators. Concluding remarks and a discussion of future work is the subject of \S \ref{sec:conclusions}. 

\section{Preliminaries} 
\label{sec:prelims}
In this section, we outline the model problem, its discretizations, and give a brief summary of MGRIT and Parareal. We then demonstrate the difficulty our seemingly simple model problem poses for these algorithms via some numerical examples.

\subsection{Model problem and discretizations}
\label{ssec:model_problem}
For the model problem, we consider the one-dimensional linear advection equation,
\begin{align} \label{eq:PDE}
\partial_t u + \alpha \partial_x u = 0, \quad (x,t) \in [-1,1] \times (0,T], \quad u(x,0) = \sin^4(\pi x),
\end{align}
with constant wave speed $\alpha > 0$. While the exact solution of this canonical hyperbolic PDE is just the shifted initial condition, and its numerical approximation is easily obtained in the sequential time-stepping setting, it presents enormous difficulty for parallel-in-time solvers. In what follows, as throughout most of this paper, we consider periodic boundary conditions in space, but briefly turn our attention to inflow/outflow boundaries in \S \ref{ssec:LSQ_inflow}.

To numerically approximate the solution of \eqref{eq:PDE}, finite-difference spatial discretizations are used with Runge-Kutta time integrators. As such, the spatial domain $x \in [-1,1]$ is discretized with $n_x+1$ equidistant points with spacing $\Delta x$, and the temporal domain $t \in [0, T]$ is discretized with $n_t+1$ equidistant points having a spacing of $\Delta t$. We employ the method of lines to generate a semi-discretized representation. First, a $p$th-order upwind finite-difference spatial discretization is applied to \eqref{eq:PDE}, resulting in the system of ordinary differential equations (ODEs)
\begin{align} \label{eq:ODEs}
\bm{u}'(t) = {\cal L} \bm{u}, 
\quad 
t \in (0,T],
\quad
\bm{u}(0) = u(\bm{x},0),
\end{align}
in which ${\cal L} \colon \mathbb{R}^{n_x} \to \mathbb{R}^{n_x}$ represents the discretization of $-\alpha \partial_x$ on the spatial mesh. Since $\alpha$ is constant and periodic boundaries are applied, ${\cal L}$ is a circulant matrix and is, thus, unitarily diagonalized by the discrete Fourier transform (DFT). Specifically, we use upwind-finite-difference spatial discretizations of orders 1--5, which we denote as U1--U5. Letting $v'_i$ denote the derivative of the single variable function $v(x)$ at point $x_i$, these are given by
\begin{align*}
\tag{U1}
v'_i  
&= 
\frac{1}{\Delta x} \big[ v_{i} - v_{i-1} \big] +
\mathcal{O}(\Delta x), 
\\
\tag{U2}
v'_i
&= 
\frac{1}{2 \Delta x} 
\big[
3v_{i} - 4v_{i-1} + v_{i-2} 
\big] 
+ 
\mathcal{O}(\Delta x^2), 
\\
\tag{U3}
v'_i
&= 
\frac{1}{6\Delta x} 
\big[
2v_{i+1} + 3v_{i} - 6v_{i-1} + v_{i-2} 
\big] 
+ 
\mathcal{O}(\Delta x^3), 
\\
\tag{U4}
v'_i
&= 
\frac{1}{12\Delta x} 
\big[
3v_{i+1} + 10v_{i} - 18v_{i-1} + 6v_{i-2} - v_{i-3} 
\big] 
+ 
\mathcal{O}(\Delta x^4), 
\\
\tag{U5}
v'_i
&= 
\frac{1}{60\Delta x} 
\big[
-3v_{i+2} + 30v_{i+1} + 20v_{i} - 60v_{i-1} + 15v_{i-2} - 2v_{i-3}
\big] 
+ 
\mathcal{O}(\Delta x^5). 
\end{align*} 
These discretizations may be constructed using standard techniques, as described by Shu\cite{Shu1998}, for example.

The ODE system \eqref{eq:ODEs} is then discretized using either a $p$th-order explicit Runge-Kutta (ERK) method, or a $p$th-order, L-stable singly diagonally implicit Runge-Kutta (SDIRK) method, with the resulting scheme denoted as either ERK$p$+U$p$, or SDIRK$p$+U$p$. We consider ERK schemes of orders 1--5 as follows: The 1st-order scheme is Euler's method; the 2nd- and 3rd-order methods are the `optimal' strong-stability-preserving schemes;\cite[(9.7), (9.8)]{Hesthaven2017} the 4th-order scheme is the `classical Runge-Kutta method';\cite[p. 180]{Butcher2003} and finally, see Butcher\cite[(236a)]{Butcher2003} for the 5th-order scheme. We consider SDIRK schemes of orders 1--4 as follows: The first-order scheme is Euler's method; the 2nd- and 3rd-order methods can be found in Butcher;\cite[pp. 261--262]{Butcher2003} and the 4th-order scheme is given by Hairer and Wanner.\cite[(6.16)]{Hairer_Wanner1996} Butcher tableaux for these Runge-Kutta schemes can be found in Appendix~\ref{sec:butcher}.

Upon application of a Runge-Kutta scheme to ODEs \eqref{eq:ODEs}, their numerical solution may be written in the one-step form
\begin{align} \label{eq:time_stepping}
\bm{u}^{n+1} = \Phi \bm{u}^{n}, \quad \bm{u}^0 = \bm{u}(0), \quad n = 0,\ldots,n_t.
\end{align}
Note that equations \eqref{eq:time_stepping} can be written as a single, large space-time block lower bidiagonal linear system. Here, $\Phi \colon \mathbb{R}^{n_x} \to \mathbb{R}^{n_x}$ is known as the `time-stepping operator,' as it steps the discrete solution from one time level to the next. The eigenvalues of $\Phi$ can be computed as a function of those of ${\cal L}$.\cite{Dobrev_etal_2017, Hessenthaler_etal_2020} In fact, it can be shown that $\Phi$ is a rational function (in a matrix sense) of ${\cal L}$, as in Lemma~\ref{lem:Phi_rational}.
\begin{lemma}[Rational form of $\Phi$]
\label{lem:Phi_rational}
Let $R(z) = P(z)/Q(z)$ denote the stability function \cite[Lemma 351A]{Butcher2003} of a Runge-Kutta scheme applied to \eqref{eq:ODEs}, in which $P$ and $Q$ are polynomials derived from the Butcher tableau of the scheme.  Then, for diagonalizable matrices ${\cal L} \in \mathbb{R}^{n_x \times n_x}$, the time-stepping operator in \eqref{eq:time_stepping} is
\begin{align} \label{eq:Phi_rational}
\Phi(\Delta t {\cal L}) = P(\Delta t {\cal L}) [Q(\Delta t {\cal L})]^{-1}.
\end{align}
\end{lemma}
\begin{proof} 
The proof essentially follows the same steps as analyzing the Runge-Kutta stability of ODE system \eqref{eq:ODEs} via decoupling it into a system of scalar ODEs, with the additional final step of recoupling the system. See, for example, Butcher,\cite[pp. 74--75]{Butcher2003} and Hairer and Wanner,\cite[pp. 15--16]{Hairer_Wanner1996} for more details.
\end{proof}
For an ERK scheme, $Q(z) = 1$ and so the Runge-Kutta stability function used in Lemma~\ref{lem:Phi_rational} is simply a polynomial, $R(z) = P(z)$. 

\begin{corollary} \label{cor:circulant_Phi}
For periodic boundary conditions applied to \eqref{eq:PDE}, the time-stepping operator $\Phi$ in \eqref{eq:time_stepping} can be written as the product of a sparse circulant matrix and the inverse of a sparse circulant matrix. In the case of an ERK scheme, $\Phi$ is simply a sparse circulant matrix.
\end{corollary}
\begin{proof}
For periodic boundaries, the finite-difference spatial discretizations ${\cal L}$ are sparse and circulant, and noting that circulant matrices are closed under addition and multiplication, the result follows immediately from the rational form of $\Phi$ in \eqref{eq:Phi_rational}.
\end{proof}

The CFL number for Runge-Kutta finite-difference discretizations of \eqref{eq:PDE} is defined as
\begin{align} \label{eq:CFL_definition}
c := \alpha \frac{\Delta t}{\Delta x}.
\end{align}
The explicit discretizations considered here suffer from a CFL limit, for which a necessary condition for numerical stability is $c \leq c_{\rm max}$. Values of $c_{\rm max}$ can be computed from the Runge-Kutta stability function and the eigenvalues of $\Delta t {\cal L}$, and are given in Table~\ref{tab:CFL_limits}. Throughout this paper, experiments using ERK discretizations will employ a CFL fraction---ratio of CFL number to CFL limit---of 85\%, $c = 0.85c_{\max}$, since it is realistic of what would be used for regular time-stepping. In all SDIRK experiments, a CFL number of $c = 4$ is used. All of the SDIRK+U schemes considered here are unconditionally stable since the eigenvalues of the circulant matrices $\Delta t {\cal L} \in \mathbb{R}^{n_x \times  n_x}$ lie in the left-half plane, which means they are contained in the stability region of any L-stable Runge-Kutta method. A wave speed of $\alpha = 1$ is used in all experiments.
\begin{table}[tbhp] 
\caption{
CFL limits $c_{\rm max}$ for ERK+U discretizations of \eqref{eq:PDE} with periodic boundary conditions. 
\label{tab:CFL_limits}
}
\begin{center}
\begin{tabular}{|c|c|c|c|c|c|} 
\hline
Scheme & 
ERK1+U1  &
ERK2+U2 & 
ERK3+U3 & 
ERK4+U4 & 
ERK5+U5 \\ \hline
$c_{\rm max}$ & 
1 & $1/2$ 
& 1.62589$\ldots$ 
& 1.04449$\ldots$
& 1.96583$\ldots$ \\ \hline
\end{tabular}
\end{center}
\end{table}

To demonstrate the accuracy of the discretizations used here and to emphasize that the high-order methods faithfully represent the non-dissipative nature of \eqref{eq:PDE}, computed discretization errors are shown in Figure~\ref{fig:discretization_errors}. The high-order methods stand in stark contrast with the first-order SDIRK1+U1 (right-hand panel), which has yet to reach its asymptotic convergence rate of one, because it possesses significant numerical diffusivity. 
\begin{figure}[h!t]
\centerline{
\includegraphics[width=0.448\textwidth]{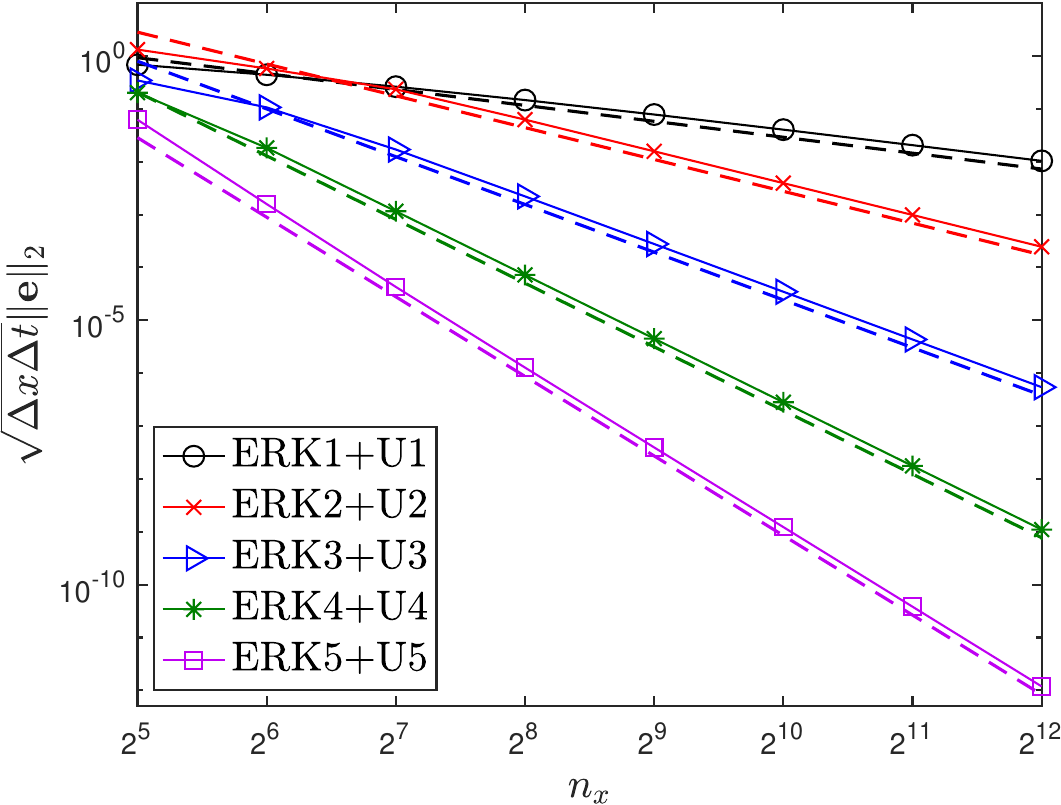}
\quad\quad
\includegraphics[width=0.448\textwidth]{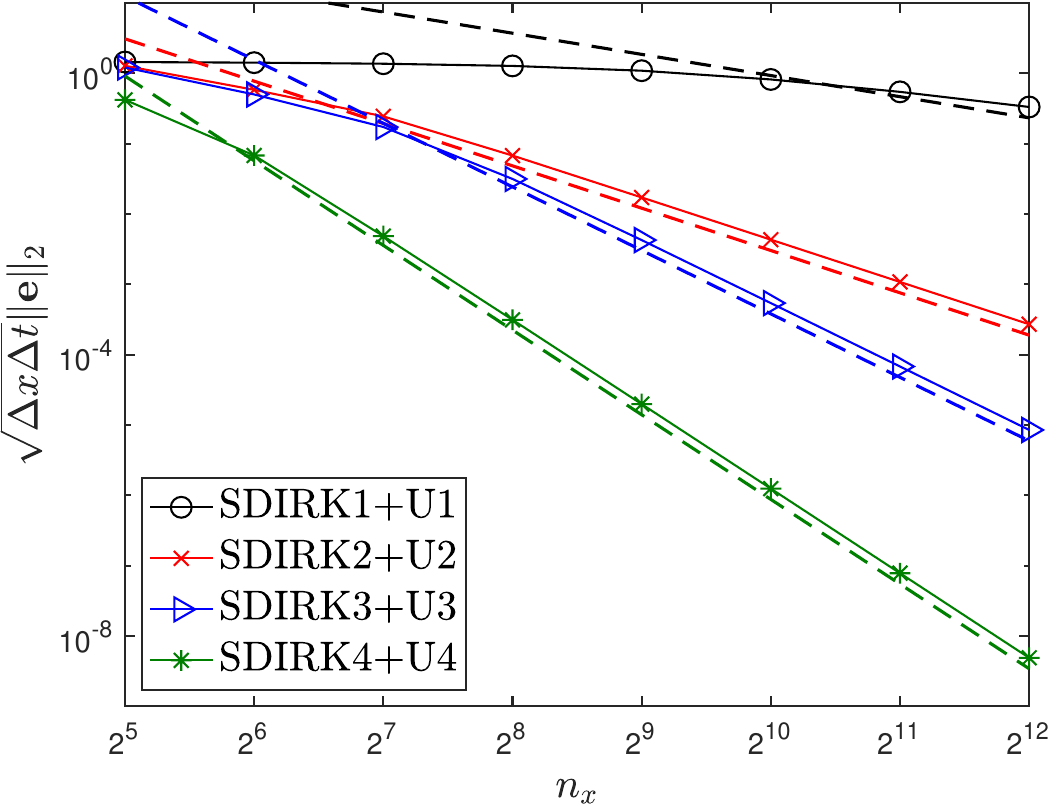}
}
\caption{
Space-time discretization errors for \eqref{eq:PDE} measured in the discrete $L_2$-norm. Left: ERK+U. Right: SDIRK+U. Plotted underneath each scheme's errors is a dashed line showing the theoretical convergence rate (order $p$ for ERK$p$/SDIRK$p$+U$p$).  Note the use of different scalings of the vertical axes in the plots.
\label{fig:discretization_errors}
}
\end{figure}

\subsection{MGRIT, Parareal, and numerical set-up}
In this section, we give a brief overview of MGRIT and Parareal, and describe the set-up for our numerical tests. MGRIT and Parareal are parallel, multilevel, iterative methods for solving block lower triangular systems arising from evolutionary problems, such as \eqref{eq:time_stepping}; note, however, that Parareal is typically only thought of as a two-level method. In this paper, we make no significant differentiation between MGRIT and Parareal, since Parareal can be described in the MGRIT framework via certain choices of algorithmic parameters (as discussed below). As such, beyond the current section, we will focus on the MGRIT algorithm.

We now give an overview of the MGRIT algorithm as it applies to model problem \eqref{eq:time_stepping}; see Falgout et al.\cite{Falgout_etal_2014, Oneill_etal_2017} or Howse et al.\cite{Howse_etal_2019} for more detailed descriptions which apply to nonlinear problems, for example. The temporal mesh $(t^n)_{n = 0}^{n_t} := (n \Delta t)_{n = 0}^{n_t}$ on which the problem is posed is the `fine grid,' and an integer coarsening factor $m > 1$ is used to induce a `coarse grid' consisting of a subset of the fine-grid points, $(m n \Delta t)_{n = 0}^{n_t/m}$. The set of points appearing exclusively on the fine grid are called `F-points,' while those appearing on both fine and coarse grids are `C-points.' An MGRIT iteration combines pre-relaxation with a coarse-grid correction. The two fundamental types of relaxation are: F-relaxation, which is time-stepping starting from each C-point across the following F-interval; and C-relaxation, which is time-stepping from the last F-point in each interval to its following C-point. The standard relaxation sweeps performed in MGRIT are either: F-relaxation (commonly used in Parareal), or the stronger FCF-relaxation, which is an F-, followed by a C-, followed by an F-relaxation. 

The coarse-grid problem is obtained by approximating the Schur complement system (with respect to C-points) of the fine-grid residual equation. This results in the algebraic error $\bm{e}_{\rm c}$ at C-points being approximated by the block lower-bidiagonal system
\begin{align} \label{eq:time_stepping_cg}
\bm{e}^{n+1}_{\rm c} = \Psi \bm{e}^{n}_{\rm c} + \bm{r}^{n+1}_{\rm c},  \quad \bm{e}^0_{\rm c} = 0, \quad n = 0,\dots, n_t/m  -1,
\end{align}
in which $\bm{r}^{n+1}_{\rm c}$ is the fine-grid residual at the $n+1$st C-point, and $\Psi \colon \mathbb{R}^{n_x} \to \mathbb{R}^{n_x}$ is the coarse-grid time-stepping operator, responsible for propagating the error from one C-point to the next. After solving \eqref{eq:time_stepping_cg}, the coarse-grid error is interpolated to the fine grid via so-called `ideal interpolation' which corresponds to injection at C-points followed by an F-relaxation. To solve the coarse-grid problem \eqref{eq:time_stepping_cg}, one can either do a sequential forward solve, or recursively apply the algorithm since \eqref{eq:time_stepping_cg} has the same block lower-bidiagonal structure as fine-grid problem \eqref{eq:time_stepping}. A sequential coarse-grid solve is almost always used in Parareal algorithms, making them two-level solvers.

Taking $\Psi = \Psi_{\rm ideal} := \Phi^m$ defines an \textit{ideal} coarse-grid time-stepping operator in the sense that the exact solution of \eqref{eq:time_stepping} is reached in a single MGRIT iteration. In this instance, coarse-grid problem \eqref{eq:time_stepping_cg} really is the Schur complement system of the fine-grid residual equation. However, no speed-up in parallel can be achieved with $\Psi = \Psi_{\rm ideal}$ since the sequential coarse-grid solve is as expensive as the original problem of time-stepping with $\Phi$ across the entire fine grid. Instead, one should choose $\Psi$ to be some approximation of $\Phi^m$---or equivalently, it should approximate taking $m$ steps with $\Phi$ on the fine grid---under the constraint that its action is significantly cheaper to compute so that speed-up can be achieved. Typically $\Psi$ is chosen through the process of rediscretizing $\Phi$ on the coarse grid, whereby it is employed with the enlarged coarse-grid time step, $m \Delta t$,\cite{DeSterck_etal_2019, Dobrev_etal_2017, Falgout_etal_2014} but other techniques have also been considered.\cite{Falgout_etal_2019, Nielsen_etal_2018}

For completeness, we now describe all of the settings used in our numerical tests. We exclusively use FCF-relaxation since we find that it generally leads to faster and more scalable solvers than using F-relaxation. The initial iterate for the space-time solution is uniformly random except at $t=0$, where it matches the prescribed initial condition. Unless otherwise noted, the metric used to report solver convergence is the number of iterations needed to achieve a space-time residual below $10^{-10}$ in the discrete $L_2$-norm. This stopping criterion exceeds the accuracy of the underlying discretizations in almost all cases, and so its use typically leads to a dramatic `over solving' of the space-time system with respect to the discretization error.  Nonetheless, we use such a small halting tolerance to highlight asymptotic convergence behavior. For all ERK$p$+U$p$ tests, a spatial resolution is selected, and a number of points $n_t$ in time is chosen to be the largest power of two such that $T = \Delta t n_t$ does not exceed 8 (note that requiring $n_t$ to be a power of two simplifies the implementation because we always coarsen by a power of two). For $p = (1,2,3,4,5)$ and a CFL fraction of 85\%, this results in final integration times $T \approx (6.8,6.8,5.5,7.1,6.7)$. For all SDIRK+U tests, $T = 8$ and $n_t = n_x$ such that a CFL number \eqref{eq:CFL_definition} of $c = 4$ results. Where scaling tests are presented, the mesh is refined uniformly in both space and time such that the CFL number of the fine-grid discretization remains constant.

\subsection{Failure of MGRIT with rediscretization for the model problem}
\label{ssec:motivating_examples}
To provide a baseline for the numerical results shown later in this paper, we now present numerical results for model problem \eqref{eq:PDE} using MGRIT with rediscretized coarse-grid operators (i.e., $\Psi$ is chosen as $\Phi$ with an enlarged time step of $m \Delta t$). For all ERK+U discretizations of \eqref{eq:PDE}, such a coarse-grid operator leads to divergent solvers for all $m$. This behavior is driven primarily by CFL instability: For coarsening factors $m > 1$, the coarse-grid CFL limit is violated (recalling fine-grid CFL numbers are set to 85\% of their respective limits), and so the resulting (unstable) coarse-grid solution cannot accelerate convergence to the (stable) fine-grid solution. To overcome this instability, a possible strategy is to couple the explicit fine-level discretization with a stable, implicit coarse-grid discretization. In such cases, a large coarsening factor is required to amortize the increased cost of solving an implicit coarse-grid problem. However, in our numerical tests (not shown here), this technique seldom results in a good solver because the approximation it provides to $\Psi_{\rm ideal} := \Phi^m$ is not good enough, even for small~$m$; furthermore, we are unaware of any results in the literature that have used this technique (either successfully or unsuccessfully) for hyperbolic problems. In the few instances where speed-up has been achieved for explicit discretizations of hyperbolic problems, alternative techniques, such as incorporating spatial coarsening,\cite{Howse_etal_2019} or reducing the order of the discretization in time\cite{Nielsen_etal_2018} have been used. While such techniques certainly avoid coarse-grid CFL instabilities, it is not clear that they result in efficient algorithms, since only small speed-ups have been observed in practice.\cite{Howse_etal_2019,Nielsen_etal_2018} Thus, we do not present numerical results for ERK discretizations here because the standard choice of rediscretization is divergent for our time-only coarsening algorithm and, to the best of our knowledge, no other technique exists for developing efficient coarse-grid operators for explicit discretizations.

In contrast to explicit discretizations, unconditionally stable, implicit fine-grid discretizations can be rediscretized on coarse grids to provide stable coarse-grid operators. Two-level MGRIT iteration counts for SDIRK+U discretizations of \eqref{eq:PDE} using such coarse-grid operators are given in the left side of Table~\ref{tab:examples_IRK_redisc}. All solvers, with the exception of SDIRK1+U1, are divergent since they converge in approximately the number of iterations for which the initial condition is sequentially propagated across the entire domain (if one uses exact arithmetic), as is done in sequential time-stepping. The relatively better performance of SDIRK1+U1 is attributable to the fact that it is highly diffusive (see Figure~\ref{fig:discretization_errors}), but it still requires a number of iterations that is much higher than what we will achieve with the new approach introduced in this paper. 

When using rediscretization, MGRIT convergence rates for implicit discretizations of hyperbolic problems strongly depend on the CFL number, with smaller CFL numbers typically resulting in faster convergence, just as in the explicit case, even though there is no CFL limit to violate. This can be seen by contrasting the types of convergence rates reported in Dobrev et al.\cite{Dobrev_etal_2017} for linear advection with those shown in Table~\ref{tab:examples_IRK_redisc} for larger, more realistic CFL numbers. This behavior stands in stark contrast with that for diffusion-dominated problems where convergence is typically achieved within 10 or so iterations, even for high-order discretizations and large coarsening factors.\cite{Dobrev_etal_2017, Falgout_etal_2014}

\begin{table}[tbhp] 
\caption{
Two-level iteration counts for SDIRK+U discretizations using a rediscretized coarse-grid operator. Left: Measured iteration counts. Right: Iteration counts at which the exact solution is achieved using exact arithmetic, $n_t/(2m)$. The `\xmark' denotes a solve which suffered an overflow error at the 358th iteration where the residual norm was approximately $10^{303}$.
\label{tab:examples_IRK_redisc}
} 
\begin{center}
\begin{tabular}{|c|c|cc|cc|} 
\hline
\multirow{2}{*}{Scheme} 
&
\multirow{2}{*}{$n_x \times n_t$} 
& 
\multicolumn{2}{c|}{Iteration count} 
& 
\multicolumn{2}{c|}{$n_t/(2m)$} \\ \cline{3-6}
&
& 
$m=2$ & $m=4$ 
&
$m=2$ & $m=4$
\\ \Xhline{2\arrayrulewidth}
\multirow{2}{*}{SDIRK1+U1} 
& $2^{10} \times 2^{10}$ & 18 & 38 & 256 & 128 \\
& $2^{12} \times 2^{12}$ & 18 & 40 & 1024 & 512 \\
\hline
\multirow{2}{*}{SDIRK2+U2} 
& $2^{10} \times 2^{10}$ & 241 & 128 & 256 & 128 \\ 
& $2^{12} \times 2^{12}$ & 1008 & 514 & 1024 & 512 \\ 
\hline
\multirow{2}{*}{SDIRK3+U3} 
& $2^{10} \times 2^{10}$ & 183 & 128 & 256 & 128 \\ 
& $2^{12} \times 2^{12}$ & 891 & 507 & 1024 & 512 \\ 
\hline
\multirow{2}{*}{SDIRK4+U4} 
& $2^{10} \times 2^{10}$ & 256 & 130 & 256 & 128 \\ 
& $2^{12} \times 2^{12}$ & \xmark & 520 & 1024 & 512 \\ 
\hline
\end{tabular}
\end{center}
\end{table}

\section{Convergence theory applied to hyperbolic problems} 
\label{sec:convergence}
To better understand the reason for poor convergence of MGRIT applied to the model problem (as shown in the previous section), and hyperbolic PDEs more generally, we now recall the two-level MGRIT convergence theory of Dobrev et al.\cite{Dobrev_etal_2017} and Hessenthaler et al.\cite{Hessenthaler_etal_2020} and discuss some of its implications.

\subsection{Two-level convergence theory}
\label{ssec:convergence}
The convergence behavior of MGRIT can be understood by analyzing its error propagation matrix, $E$. That is, if $\bm{e}^{(0)}$ is the initial space-time error, then after $q$ MGRIT iterations, the error obeys $\Vert \bm{e}^{(q)} \Vert \leq \Vert E \Vert^q \Vert \bm{e}^{(0)} \Vert$. To analyze $\Vert E \Vert$, we assume that the fine-grid time stepper $\Phi$ and coarse-grid time stepper $\Psi$ are simultaneously diagonalizable by a unitary transform, and denote their eigenvalues by $(\lambda_k)_{k = 1}^{n_x}$, and $(\mu_k)_{k = 1}^{n_x}$, respectively. These assumptions are satisfied by the $\Phi$ and $\Psi$ considered here, since all circulant matrices are unitarily diagonalized by the DFT. Furthermore, the eigenvalues should satisfy $|\lambda_k|, |\mu_k| < 1 \, \forall \, k$, so that the time-stepping methods are stable; note $|\lambda_k| = |\mu_k| = 1$ is permissible, but the following result does not apply for such cases.

The assumption of diagonalizability allows error reduction for each spatial mode to be considered individually. That is, if $E_k$ is the error propagator associated with the $k$th spatial mode, then we can consider $\Vert E_k \Vert$ for each $k$. Moreover, $\Vert E \Vert_2 = \max_k \Vert E_k \Vert_2$. For the case of FCF-relaxation considered in this paper, the 2-norm of $E_k$ can be bounded as,\cite[Lemma 4.1, Theorem 4.3]{Hessenthaler_etal_2020}\textsuperscript{,}\cite[Theorem 3.3]{Dobrev_etal_2017}
\begin{align} \label{eq:Dobrev_bounds}
\left\Vert
E_{k} 
\right\Vert_2
\leq
\sqrt{m} \,
| \lambda_k |^m
\frac{| \lambda_k^m - \mu_k | }
{ 1 - | \mu_k | }
\left( 1 - \left| \mu_k \right|^{n_t/m-1} \right)
\end{align}
Southworth\cite{Southworth2019} showed that this bound is equal to $\Vert E_k \Vert_2$ up to ${\cal O}(m/n_t)$, and, so, it is sharp.

Given bound \eqref{eq:Dobrev_bounds}, the question is now: What is required of $\Psi$, by way of its eigenvalues $(\mu_k)_{k = 1}^{n_x}$, for fast MGRIT convergence?  Note that under the assumption $|\mu_k| < 1$, the last factor $1 - | \mu_k |^{n_t/m-1} \to 1$ as $n_t \to \infty$, meaning convergence is primarily determined by the preceding factors. Firstly, convergence of the $k$th mode is related to how closely $\mu_k \approx \lambda_k^m$. So, in general, the spectrum of $\Psi$ should approximate that of $\Phi^m$ (but recall $\Psi = \Phi^m$ is not practically feasible). Secondly, from the denominator, error modes associated with $|\mu_k| \approx 1$ are potentially damped much slower than those having $|\mu_k| \ll 1$. This slow convergence must be rectified by ensuring the approximation $\mu_k \approx \lambda_k^m$ is more accurate for these modes. Thus, $\Psi$ must most accurately approximate the largest (in magnitude) eigenvalues of $\Phi^m$. As noted by Southworth\cite{Southworth2019}, the largest (in magnitude) eigenvalues of $\Phi$ typically correspond to the smoothest spatial modes, and, so, equivalently, the action of $\Psi$ must most accurately approximate that of $\Phi^m$ for spatially smooth modes.
Moreover, the leading $|\lambda_k|^m$ factor provides an additional damping mechanism for modes having $|\lambda_k| \ll 1$; note this factor arises as a consequence of using FCF- rather than F-relaxation.\cite{Dobrev_etal_2017, Hessenthaler_etal_2020} In summary, fast convergence necessitates the approximation $\mu_k \approx \lambda_k^m$ to hold $\forall k$, and with increasing accuracy for $|\lambda_k|,|\mu_k|  \to 1$.

\subsection{Implications of convergence theory}
We now provide some insight as to why MGRIT convergence is typically much worse for advection-dominated problems compared with their diffusion-dominated counterparts. Discretizations of advection-dominated PDEs are (usually) much less dissipative than discretizations of diffusion-dominated PDEs since the PDEs themselves have little dissipation (or none in the hyperbolic limit). The amount of dissipation of the $k$th spatial mode for a given discretization $\Phi$ is directly related to the value of $|\lambda_k|$. Typically, for a discretization of a diffusion-dominated problem, there are very few $|\lambda_k| \approx 1$, and many $|\lambda_k| \ll 1$, while for an accurate discretization of an advection-dominated problem, there are many more $|\lambda_k| \approx 1$. In either case, $|\lambda_k| \approx 1$ typically correspond to smooth spatial modes. This behavior is seen in the top row of Figure~\ref{fig:Phi_Psi_eigs}, where the (square of the) eigenvalues of $\Phi$ for a purely diffusive and a purely advective PDE are shown. 

\begin{figure}[t]
\centerline{
\includegraphics[width = 0.405\textwidth]{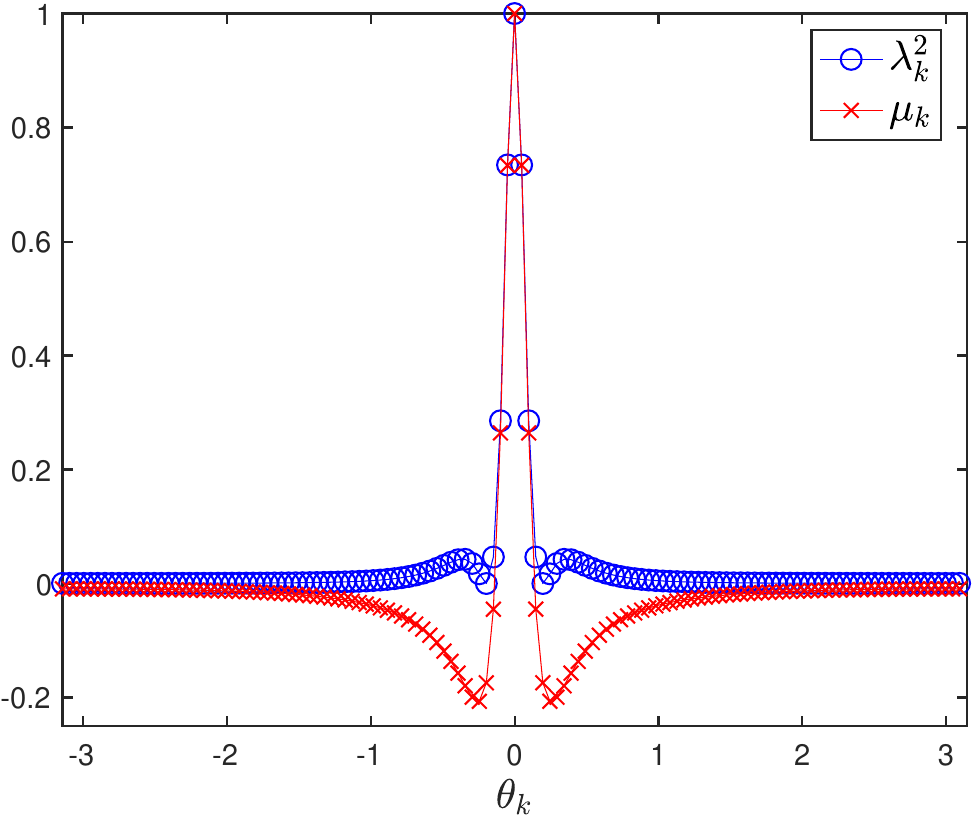}
\quad\quad
\includegraphics[width = 0.35\textwidth]{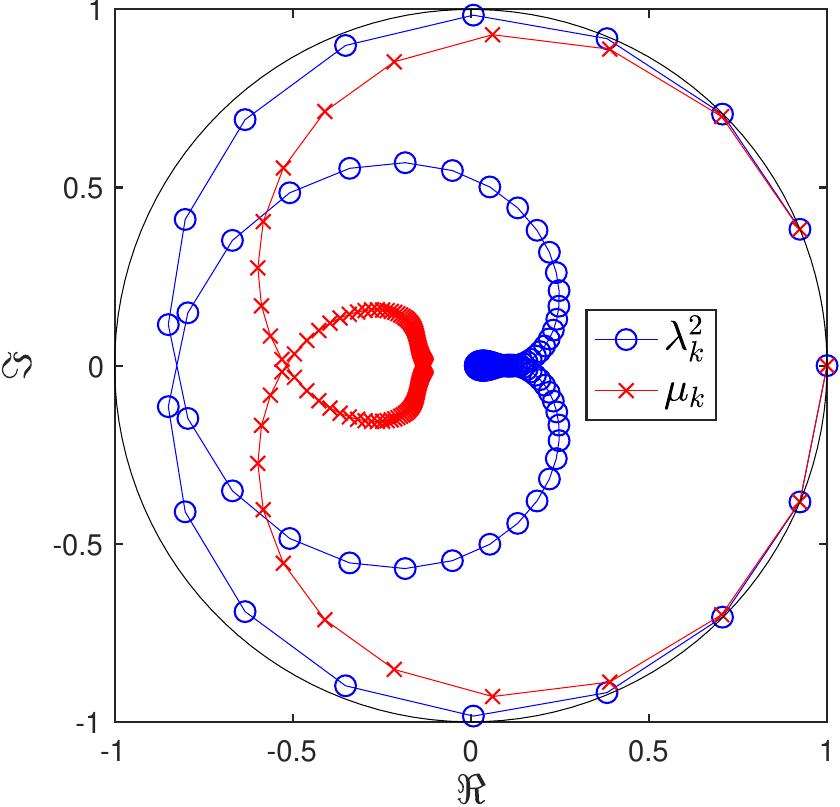}
}
\vspace{2ex}
\centerline{
\includegraphics[width = 0.412\textwidth]{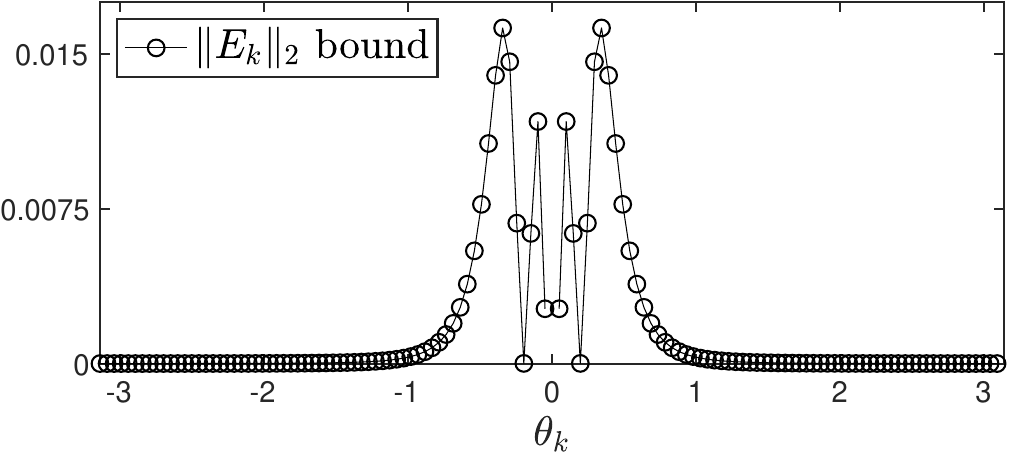}
\quad\quad
\includegraphics[width = 0.4\textwidth]{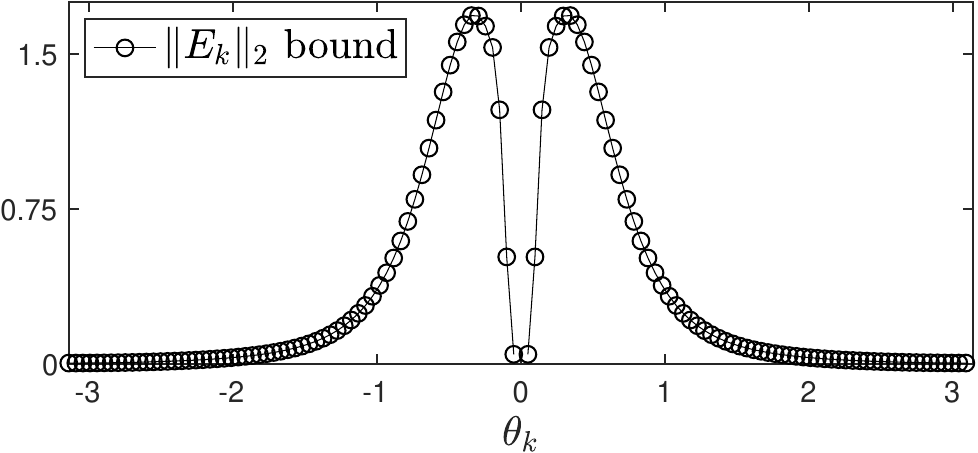}
}
\caption{
Left column: Diffusion equation $\partial_t u = \partial_{xx} u$ discretized with SDIRK2 in time and 2nd-order central finite-differences in space. Right column: Advection equation $\partial_t u + \partial_x u = 0$ discretized with SDIRK3+U3. Top row: Eigenvalues $\lambda_k^2$ of $\Phi^2$, and $\mu_k$ of $\Psi$, with $\Psi$ defined by rediscretizing $\Phi$ with $m = 2$. Bottom row: Error bound \eqref{eq:Dobrev_bounds} for each problem as a function of Fourier frequency, $\theta_k$. Both problems are subject to periodic boundary conditions in space, and are discretized on a space-time mesh covering $(x,t) \in (-1,1) \times [0,8]$ having $\Delta t = \Delta x = 1/64$, so that $\Phi, \Psi \in \mathbb{R}^{128 \times 128}$. \label{fig:Phi_Psi_eigs}}
\end{figure}

Since diffusion-dominated problems have so few $|\lambda_k| \approx 1$, $\Psi$ only has to accurately approximate very few eigenvalues of $\Phi^m$ to yield fast convergence. Conversely, since advection-dominated problems have many $|\lambda_k| \approx 1$, and very few $|\lambda_k| \ll 1$, $\Psi$ has to accurately approximate a much greater proportion of the eigenvalues of $\Phi^m$. In general, this makes the task of identifying a good $\Psi$ more difficult since it should have a simpler structure than $\Phi^m$ so that its action is less expensive to compute. 

The properties just discussed, in conjunction with the plots in Figure~\ref{fig:Phi_Psi_eigs}, help to illuminate why rediscretization of $\Phi$ with time step $m \Delta t$ typically leads to a good $\Psi$ for diffusion-dominated problems, but is often a poor choice for advection-dominated problems. 
In Figure~\ref{fig:Phi_Psi_eigs}, notice the largest eigenvalues of $\Phi^2$ are clustered around spatial frequency $\theta_k = 0$, noting that the $\theta_k = 0$ eigenvalue is at position (1,0) in the top right panel. In each instance, we see that $\mu_k$ provides a good approximation to $\lambda_k^2$ for the smoothest modes, $\theta_k \approx 0$. For the diffusion problem, this approximation is adequate to obtain fast convergence as the decay of $|\lambda_k|^2$ away from $\theta_k = 0$ is very rapid. However, for the advection problem, the approximation is inadequate since the decay of $|\lambda_k|^2$ away from $\theta_k = 0$ is more gradual and, so, the mismatch between $\lambda_k^2$ and $\mu_k$ is of much greater detriment for convergence. Error bounds \eqref{eq:Dobrev_bounds} are shown in Figure~\ref{fig:Phi_Psi_eigs} underneath the eigenvalue plots for each problem. In the diffusion case, the bound is very small and, so, convergence is fast while, in the advection case, the bound exceeds one, indicating that the solver will be divergent.  It is the smooth modes that are not accurately captured by $\Psi$ that cause the most issue. 

In summary, rediscretization essentially fails for advection-dominated problems since it does not provide an adequate approximation to $\Phi^m$ for smooth spatial modes, which tend to decay more slowly in time than for diffusion-dominated problems. Similar ideas were identified by Yavneh\cite{Yavneh1998} as being responsible for the breakdown of geometric multigrid for advection-dominated, steady-state PDEs. There, Fourier analysis showed that an inadequate coarse-grid correction of some asymptotically smooth error modes, so-called `characteristic components,' is responsible for poor multigrid performance. It is conceivable that the problematic modes we identify here actually correspond to such space-time characteristic components, but lacking an analysis similar to that by Yavneh\cite{Yavneh1998}, this is difficult to say; this is an area of current research. Nonetheless, it is likely that some of the ideas proposed by Yavneh\cite{Yavneh1998} for improving spatial multigrid solvers will be useful for developing improved MGRIT solvers.

\section{Coarse-grid operators based on a linear approximation of $\Psi_{\rm \lowercase{ideal}}$}
\label{sec:Psi_linear_approx}
From the discussion surrounding error estimates \eqref{eq:Dobrev_bounds}, the coarse-grid operator $\Psi$ should approximately minimize the difference between its spectrum and that of $\Phi^m$, in general, and particularly for larger $|\lambda_k|$. To this end, we consider $\Psi$ as the solution of the minimization problem
\begin{align} \label{eq:LSQ_general}
\Psi 
:=
\argmin \limits_{\widehat{\Psi} \in \mathbb{R}^{n_x \times n_x}} 
\Big\Vert 
W^{1/2}_{\bm{\lambda}} 
\Big[ \bm{\lambda}^m - \bm{\mu} \big( \widehat{\Psi} \big) \Big] 
\Big\Vert_2^2,
\end{align}
where $\bm{\lambda} = \big( \lambda_1, \ldots, \lambda_{n_x} \big)^\top$, $\bm{\mu} = \big( \mu_1, \ldots, \mu_{n_x} \big)^\top \in \mathbb{C}^{n_x}$, and $(\bm{\lambda}^m)_k \equiv \lambda_k^m$. Here, $W_{\bm{\lambda}} := \diag(\bm{w}) \in \mathbb{R}^{n_x \times n_x}$ is a weighting matrix, whose $k$th entry is $w_k := w(|\lambda_k|)$, in which $w \colon \mathbb{R}_+ \to \mathbb{R}_+$ is a weighting function reflecting the heuristic that it is most important to minimize $\lambda_k^m - \mu_k$ for $|\lambda_k|, |\mu_k| \to 1$, and less important for $|\lambda_k|,|\mu_k| \ll 1$. One choice of weighting function that we have found to yield good results is
\begin{align*} 
w(z) = \frac{1}{(1 - z+ \epsilon)^2},
\end{align*}
with $0 < \epsilon \ll 1$ a constant to avoid division by zero; we take $\epsilon = 10^{-6}$ in the numerical results shown here.  Note that allowing a free choice for $\widehat{\Psi}$ would naturally result in the choice $\Psi = \Phi^m$ and, so, the optimization in \eqref{eq:LSQ_general} is constrained by a pre-specified sparsity pattern of $\Psi$.

In general the eigenvalues of matrix $\Psi$ depend nonlinearly on its entries and thus \eqref{eq:LSQ_general} constitutes a nonlinear minimization problem. However, for the special case of circulant $\Psi$, \eqref{eq:LSQ_general} reduces to a linear least squares problem because the eigenvalues of a circulant matrix depend only linearly on its entries (they are given by the DFT---a linear operator---applied to its first column).
For explicit temporal discretizations of model problem \eqref{eq:PDE} with periodic boundaries, it is reasonable to impose that $\Psi$ is a sparse circulant matrix because $\Phi$ is and, hence, so too is $\Phi^m$ (Corollary~\ref{cor:circulant_Phi}). In \S~\ref{ssec:implicit}, we will also show that a sparse $\Psi$ may be used with implicit temporal discretizations of \eqref{eq:PDE}, for which $\Phi$ and $\Phi^m$ are dense matrices. Thus, for the remainder of this section, we focus exclusively on the case in which $\Psi$ is a sparse circulant matrix such that \eqref{eq:LSQ_general} reduces to a linear least squares problem.

We have also formulated and solved a nonlinear least squares problem that is based on a more direct minimization of error estimates \eqref{eq:Dobrev_bounds} than the heuristic-based \eqref{eq:LSQ_general}. This more elaborate approach, however, gives similar results for the simple model problem considered here and so it has been omitted for brevity.

\subsection{Linear least squares formulation}
Let $\tilde{\bm{\phi}}^m,\tilde{\bm{\psi}} \in \mathbb{R}^{n_x}$
denote the first columns of the circulant matrices $\Phi^m$ and $\Psi$, respectively, and note that a circulant matrix can be fully specified by its first column. Assuming the sparsity pattern of $\Psi$ is given, we let $R \in \mathbb{R}^{\nu \times n_x}$ be the restriction operator that selects these $\nu$ non-zero entries from $\tilde{\bm{\psi}}$, where $\nu \ll n_x$ since the column is sparse. Further details on the choice of $R$ (or equivalently, the sparsity pattern of $\Psi$) will be given in the following sections. The unknowns are thus the non-zero components of $\tilde{\bm{\psi}}$, which we denote as
$\bm{\psi} 
:= 
R \tilde{\bm{\psi}} 
\in 
\mathbb{R}^{\nu}$. 
Finally, let ${\cal F} \in \mathbb{C}^{n_x \times n_x}$ be the DFT matrix, then by the properties of circulant matrices, $\bm{\lambda}^m = {\cal F} \tilde{\bm{\phi}}^m$, and $\bm{\mu} = {\cal F} \tilde{\bm{\psi}} = {\cal F} R^\top \bm{\psi}$ since $R^\top R$ has ones on the diagonal in rows where $\tilde{\bm{\psi}}$ has non-zeros and zeros everywhere else. Thus, \eqref{eq:LSQ_general} can be written as a linear least squares problem for the non-zero entries in the first column of $\Psi$:
\begin{align} \label{eq:LSQ_1D_problem}
\bm{\psi} 
:= 
\argmin \limits_{\hat{\bm{\psi}} \in \mathbb{R}^{\nu}}  
\left\Vert 
W^{1/2}_{\bm{\lambda}} 
{\cal F}
\left(
\tilde{\bm{\phi}}^m - R^\top \hat{\bm{\psi}}
\right)
\right\Vert_2^2.
\end{align}

\begin{remark}[Minimizing $\Vert \Phi^m - \Psi \Vert_2^2$ is not sufficient to obtain a good solver]
\label{rem:unit_weight}
For weighting function $w = 1$, or $W_{\bm{\lambda}} = I$, \eqref{eq:LSQ_1D_problem} corresponds to minimizing the difference between the spectra of $\Phi^m$ and $\Psi$ in the two-norm. This is equivalent to minimizing the difference between $\Phi^m$ and $\Psi$ in the two-norm since they are both diagonalized by the same unitary transform. In this instance, the solution of \eqref{eq:LSQ_1D_problem} can be computed explicitly as $\bm{\psi} = R \tilde{\bm{\phi}}^m$, which means that $\Psi$ is given by truncating $\Phi^m$ in the sparsity pattern of $\Psi$. We have found that this choice of $\Psi$ typically does not lead to a fast or scalable solver for model problem \eqref{eq:PDE} because it does not adequately capture the dominant eigenvalues of $\Phi^m$ (see \S \ref{sec:convergence}). 
\end{remark}

\begin{lemma} \label{lem:LSQ_1D_real_sol}
The solution of \eqref{eq:LSQ_1D_problem} is real valued.
\end{lemma}
\begin{proof}
The normal equations of \eqref{eq:LSQ_1D_problem} are
\begin{align*}
\left( 
R {\cal F}^* W_{\bm{\lambda}} {\cal F} R^\top 
\right)
\bm{\psi} 
= 
\left( R {\cal F}^* W_{\bm{\lambda}} {\cal F} \right) \tilde{\bm{\phi}}^m.
\end{align*}
Since $R$ and $\tilde{\bm{\phi}}^m$ are real, $\bm{\psi}$ is real if the circulant matrix ${\cal A} := {\cal F}^* W_{\bm{\lambda}} {\cal F} = {\cal F}^* \diag(\bm{w}) {\cal F}$ is real. Letting $\tilde{\bm{a}}$ denote the 1st column of ${\cal A}$, then, because ${\cal A}$ is circulant, $\tilde{\bm{a}} = {\cal F}^* \bm{w}$; that is, $\tilde{\bm{a}}$ is the inverse DFT of $\bm{w}$. Appealing to properties of the inverse DFT, since $\bm{w}$ is real, $\tilde{\bm{a}}$ will be real if $\bm{w}$ has even symmetry, meaning that $w_{k} = w_{n_x-k}$. Using the explicit formula for the eigenvalues of circulant matrices, it is easy to verify that eigenvalues $\lambda_k$ of any real-valued circulant matrix $\Phi$ must satisfy $|\lambda_k|^2 = \lambda_k \lambda_{k}^* = \lambda_k \lambda_{-k} = \lambda_k \lambda_{n_x-k} = |\lambda_{n_x-k}|^2$. It follows that $\bm{w}$ is even since $w_k = w(|\lambda_k|) = w(|\lambda_{n_x-k}|) = w_{n_x - k}$ and, thus, ${\cal A}$ is real.
\end{proof}

In practice, the numerical solution of \eqref{eq:LSQ_1D_problem} is found to have some small imaginary components since ${\cal F}$ is complex and the problem is ill-conditioned. We simply truncate these components from the solution, as is justified by Lemma~\ref{lem:LSQ_1D_real_sol}. In some cases, the imaginary component can become large and simply truncating them from the solution has never been found to result in a good solver; see Table \ref{tab:LSQlin_ERK}. This is also observed for some other choices of the weight matrix $W$ leading to particularly ill-conditioned matrices in \eqref{eq:LSQ_1D_problem}. In practice, if an imaginary component larger than $10^{-8}$ is detected, we flag the results and do not accept the resulting $\Psi$ as a coarse-grid operator. We note, however, that this does not happen for the sparsity patterns of $\Psi$ that we advocate in the following sections.

\subsection{Explicit schemes: Selection of $\Psi$'s non-zero pattern}
\label{ssec:sparsity_selection}
Before solving \eqref{eq:LSQ_1D_problem}, we must first decide how to constrain the non-zero pattern of $\Psi$.
Our goal is to develop coarse-grid operators $\Psi$ that result in convergence in a small number of multigrid iterations, but that are sufficiently sparse to obtain a low cost per iteration. In multigrid, the cost per iteration is quantified by the so-called operator complexity,
which is defined in the case of MGRIT as the total amount of work done in time-stepping on all levels, relative to the time-stepping work on the finest level. The operator complexity depends on the sparsity of the coarse-grid operators $\Psi$.
Clearly, we require $\Psi$ to be significantly sparser than the exact coarse-grid operator $\Phi^m$,
so that time-stepping on the coarse grid is
substantially less expensive than on the fine grid.
Ideally, we would like the coarse operator $\Psi$ not to be denser than the fine-level operator $\Phi$, as would
result from rediscretizing the PDE on the coarse grid (where we note that the exact coarse propagator
$\Phi^m$ is much denser than $\Phi$), but our numerical results will show that constraining $\Psi$ to have as
few non-zeros as $\Phi$ does not yield good solvers in general. 
Still, it is useful to consider the case where $\texttt{nnz}(\Psi)=\texttt{nnz}(\Phi)$, with $\texttt{nnz}(A)$
denoting the number of non-zeros of matrix $A$, as a reference case that we will compare the per-iteration cost
of our operators $\Psi$ with.
To compute the operator complexity, let $\Phi_{\ell}$ denote the time-stepping operator on level $1 \leq \ell \leq L$ of a multilevel hierarchy with $L > 1$ levels, meaning that $\Phi \equiv \Phi_1$ and $\Psi \equiv \Phi_2$ in the two-level notation we have been using so far. Now, assuming $\Phi_{\ell}$ is a sparse operator, the work required to time-step with it is proportional to $\texttt{nnz}(\Phi_{\ell}) n_t m^{1 - \ell}$, assuming a constant coarsening factor of $m$ on all levels. Thus, the operator complexity is given by
\begin{align} \label{eq:OC}
\textrm{operator complexity} := \frac{1}{\texttt{nnz}(\Phi_1)} \sum \limits_{\ell = 1}^{L} m^{1-\ell} \texttt{nnz}(\Phi_{\ell}).
\end{align}
An efficient multigrid cycle should have an operator complexity that is bounded independently of $L$ (so that the cost of the work on the coarse levels relative to the fine-level work is bounded by a constant independent of the problem size and the number of levels). In fact, if one uses ideal coarse-grid operators, $\Phi_{\ell+1} = \Phi_{\ell}^m\, \forall \ell$, the solver will have an operator complexity equal to $L$. In contrast, in our reference case where $\texttt{nnz}(\Psi) = \texttt{nnz}(\Phi)$, a two-level solver has an operator complexity bounded by $1 + 1/m$, and a multilevel solver obeying this condition on all levels has a complexity bound of $1 + 1/(m-1)$. In the following sections, we will compare measured operator complexities for the coarse operators $\Psi$ we derive
with these reference complexities.

Next, we discuss how to choose the locations of the non-zeros in $\Psi$. 
Note that simply rediscretizing $\Phi$ on a temporally coarsened mesh leads to $\Psi$ having the same non-zero pattern as $\Phi$.
To motivate a better choice of sparsity pattern, we consider the effects of temporal coarsening on the exact solution of \eqref{eq:PDE} when it is sampled on a space-time mesh; a schematic diagram of this example is shown in Figure~\ref{fig:characteristic_coarsening}. The solution of a hyperbolic PDE is propagated through space-time along its characteristics, $x(t)$. Advection problem \eqref{eq:PDE} simply has characteristics that are straight lines with slope $\d x/\d t = \alpha$. Now, say we have an exact fine-grid time-stepping operator, $\Phi_{\rm exact}$, that advects the PDE solution along characteristics from one time level to the next. From the diagram, it is clear that $\Phi_{\rm exact}$ propagates the solution not only a distance of $\Delta t$ in time, but also a distance of $\Delta x$ in space. Considering semi-coarsening in time, by a factor of $m = 4$, for example, the resulting exact coarse-grid time-stepping operator is $\Psi_{\rm exact} = \Phi_{\rm exact}^4$. By definition, $\Psi_{\rm exact}$ propagates the solution forward in time by a distance of $4 \Delta t$; however, we see that it also propagates the solution a distance of $4 \Delta x$ in space. Thus, coarsening in the time direction, but not in space, has shifted the spatial stencil of $\Psi_{\rm exact}$ (which reaches back four points in space) with respect to that of $\Phi_{\rm exact}$ (which reaches back one point in space).
\begin{figure}[h!t]
\centering
\includegraphics[width=0.6\linewidth]{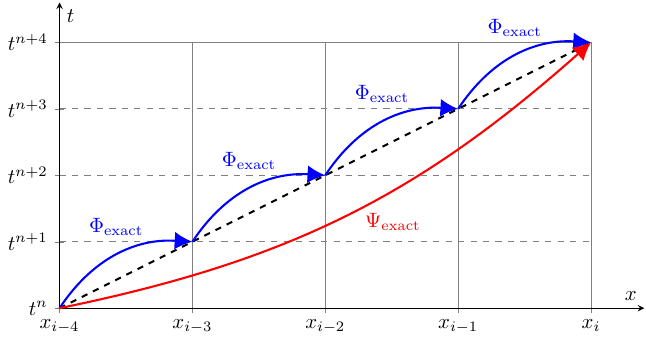}
\caption{
Exact fine- and coarse-grid time-stepping operators, $\Phi_{\rm exact}$ and $\Psi_{\rm exact}$, propagate the solution of $\partial_t u + \alpha \partial_x u = 0$ along one of its characteristics (thick, dashed line) on a fine and coarse grid, respectively. The fine grid has a temporal mesh spacing of $\Delta t = \Delta x/\alpha $, and the coarse grid of $4 \Delta t$, with coarse-grid points at $t^{n}$ and $t^{n+4}$.
\label{fig:characteristic_coarsening}
}
\end{figure}

From an algebraic perspective, it seems reasonable to consider $\Psi \approx \Phi^m$ having its sparsity pattern based on the largest non-zeros of $\Phi^m$. To assess this, we compute $\Phi^m$ and examine its non-zeros as a function of their diagonal index $i$ (recall entries of $\Phi^m$ are constant along its diagonals since it is circulant). We define diagonal index $i$ to be 0 on the main diagonal, negative below the main diagonal, and positive above the main diagonal. For $m \in \{16,64\}$, these are shown in Figure~\ref{fig:Psi_ideal_entries}. There is clearly a well-defined distribution in the magnitude of these non-zeros for each scheme. The distributions peak at different $i$ essentially because the time step is chosen differently for each scheme, since $c = 0.85 c_{\max}$ and $c_{\max}$ is different for each scheme (see Table~\ref{tab:CFL_limits}). In the plots, dashed lines represent $mc \Delta x$, which is the spatial distance travelled along a characteristic departing from $t^n$ and arriving at $t^{n+m}$. This illustrates, not unexpectedly, that the discretizations provide an approximation to the advection of the solution along characteristics that occurs at the PDE level. 

\begin{figure}[h!t]
\centerline{
\includegraphics[width = 0.44\textwidth]{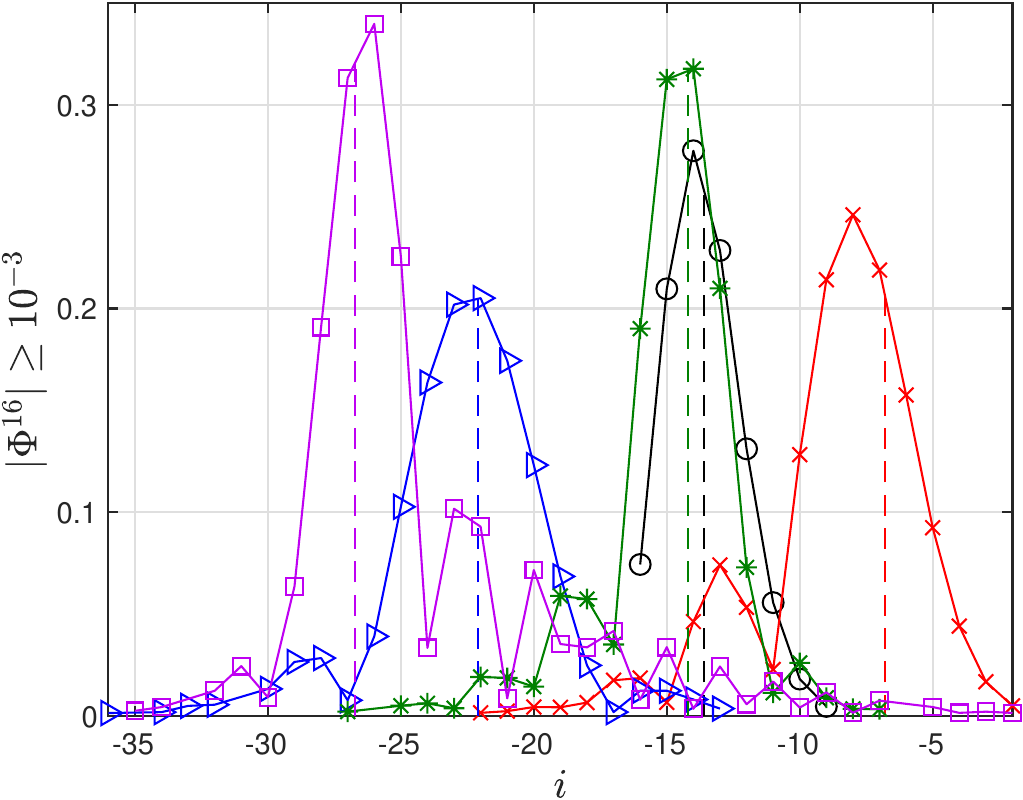}
\quad\quad
\includegraphics[width = 0.448\textwidth]{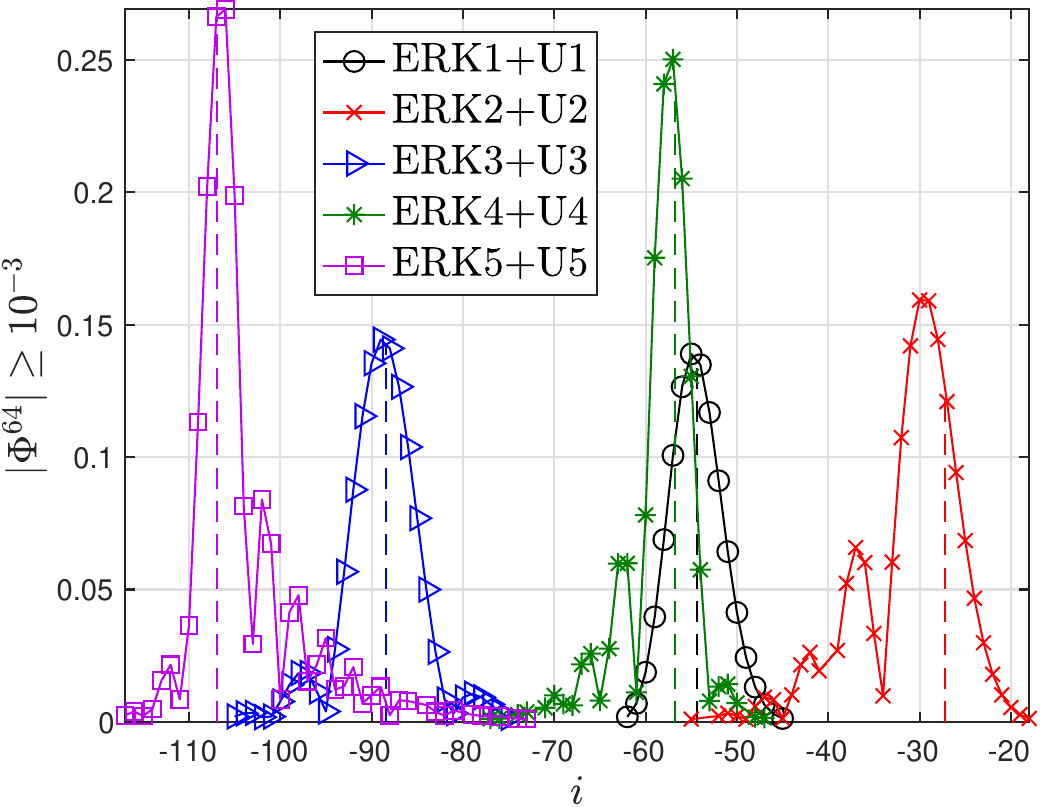}
}
\caption{
Magnitude of diagonal entries of $\Psi_{\rm ideal} := \Phi^m$, $m =16$ (left), and $m=64$ (right), that are larger than $10^{-3}$, as a function of their diagonal index, $i$. Fine-level discretizations, $\Phi$, are ERK$p$+U$p$, $p \in \{1,2,3,4,5\}$. Dashed vertical lines for each discretization are included to represent a distance of $-mc \Delta x$ from $i = 0$. Note that $c = 0.85 c_{\rm max}$ is different for each scheme. 
\label{fig:Psi_ideal_entries}
}
\end{figure}

From our previous intuitive arguments involving Figure~\ref{fig:characteristic_coarsening}, it is clear that the non-zero pattern of $\Psi$ should reflect the characteristic nature of the PDE; this is also the conclusion reached from an algebraic perspective of approximating $\Phi^m$ (Figure~\ref{fig:Psi_ideal_entries}). We note that Gander\cite{Gander2008} also argued that making use of characteristic information may be important. The specific sparsity patterns used for the ERK+U schemes will be discussed further in \S\ref{ssec:LSQ_results}.

\subsection{Explicit schemes: Two-level results}
\label{ssec:LSQ_results}
In this section, we show MGRIT convergence with $\Psi$ as the solution of least squares problem \eqref{eq:LSQ_1D_problem} for ERK discretizations. To demonstrate the validity of the ideas outlined in the previous section, we solve the least squares problem for $\Psi$ having a sparsity pattern equal to that of the fine-level operator $\Phi$, and for it having a sparsity pattern based on the
largest nonzeros of $\Phi^m$. We will first compare numerical results for these two approaches, and will then provide
details on how exactly we choose a sparsity pattern based on the largest nonzeros of $\Phi^m$.

Solver iteration counts for $\Psi$ having the same sparsity pattern as $\Phi$ are shown in the left side of Table~\ref{tab:LSQlin_ERK}. A convergent solver was not obtained for any $m$ for the 1st-order scheme, convergent schemes were obtained only for $m = 2$ for the 2nd- and 3rd-order schemes, and convergent schemes with $m \in \{2,4,8\}$ were obtained for the 4th- and 5th-order schemes. In all cases where convergent solvers were found, the iteration counts remain constant as the mesh is refined. For the cases where the solvers converge, these results are certainly an improvement on those using rediscretization, which are divergent in this setting due to coarse-level CFL instability (see \S\ref{ssec:motivating_examples}), attesting to the power of the optimization approach. However, for many coarsening factors and discretizations, the results are significantly worse than those obtained when using a sparsity pattern based on $\Phi^m$, as shown in the right side of Table~\ref{tab:LSQlin_ERK}. 
When considering the magnitude of entries in $\Phi$ and $\Phi^m$ for $m = 2,4,8$ (plots not shown here for brevity), the locations of non-zeros in $\Phi$ correspond primarily to the largest non-zeros in $\Phi^m$ for the cases in which the $\Phi$-based sparsity patterns yield convergent solvers.
Furthermore, least squares problem \eqref{eq:LSQ_1D_problem} was severely ill-conditioned for many instances in which $\Psi$ and $\Phi$ shared a sparsity pattern (see Table~\ref{tab:LSQlin_ERK}), but never when $\Psi$ and $\Phi^m$ shared a similar sparsity pattern, further supporting our argument that using a characteristic-based sparsity pattern is the better choice.
 
\begin{table}[tbhp] 
\caption{
Two-level iteration counts for ERK+U discretizations with $\Psi$ as linear least squares solution \eqref{eq:LSQ_1D_problem}. Left: Sparsity pattern of $\Psi$ is equal to that of $\Phi$. Right: Sparsity pattern of $\Psi$ is based on that of $\Phi^m$. An `\xmark' denotes a solve that did not converge to the required tolerance in significantly fewer than $n_t/(2m)$ iterations (i.e., the number of iterations at which the exact solution is reached). An `\xmark\textsuperscript{*}' denotes a solve in which the least squares solution had imaginary components larger than $10^{-8}$, as another indication of divergence, and an indication that the least squares problem was severely ill-conditioned.
\label{tab:LSQlin_ERK}
}
\begin{center}
\begin{tabular}{|c|c|cccccc V{3} cccccc|} 
\hline
\multirow{2}{*}{Scheme} 
&
\multirow{2}{*}{$n_x \times n_t$} 
& 
\multicolumn{6}{c V{3} }{$m$\; ($\Phi$-based sparsity)}
&
\multicolumn{6}{c|}{$m$\; ($\Phi^m$-based sparsity)} \\  \cline{3-14} 
&
& 
2 & 4 & 8 & 16 & 32 & 64
&
2 & 4 & 8 & 16 & 32 & 64
\\ 
\Xhline{2\arrayrulewidth}
\multirow{3}{*}{ERK1+U1} 
& $2^{8\phantom{1}} \times 2^{10}$ & 
\xmark & \xmark & \xmark & \xmark & \xmark & \xmark &
11 & 6 & 6 & 7 & 6 & 5
\\
& $2^{10} \times 2^{12}$ & 
\xmark & \xmark & \xmark & \xmark & \xmark & \xmark &
11 & 6 & 6 & 7 & 6 & 5
\\
& $2^{12} \times 2^{14}$ & 
\xmark & \xmark & \xmark & \xmark & \xmark & \xmark &
11 & 6 & 6 & 7 & 6 & 5
\\
\hline
\multirow{3}{*}{ERK2+U2} 
& $2^{8\phantom{1}} \times 2^{11}$ & 
10 & \xmark & \xmark & \xmark & \xmark & \xmark &
10 & 7 & 9 & 8 & 7 & 7
\\
& $2^{10} \times 2^{13}$ & 
10 & \xmark & \xmark & \xmark & \xmark & \xmark &
10 & 7 & 9 & 8 & 7 & 7
\\
& $2^{12} \times 2^{15}$ & 
10 & \xmark & \xmark & \xmark & \xmark & \xmark &
10 &  7 & 9 & 8 & 7 & 7
\\\hline
\multirow{3}{*}{ERK3+U3} 
& $2^{8\phantom{1}} \times 2^{9\phantom{1}}$ 
& 9 & \xmark & \xmark\textsuperscript{*} & \xmark\textsuperscript{*} & \xmark\textsuperscript{*} & \xmark\textsuperscript{*} &
7 & 6 & 5 & 6 & 5 & 3
\\
& $2^{10} \times 2^{11}$ 
& 9 & \xmark & \xmark\textsuperscript{*} & \xmark\textsuperscript{*} & \xmark\textsuperscript{*} & \xmark\textsuperscript{*} &
7 & 6 & 5 & 6 & 5 & 4
\\
& $2^{12} \times 2^{13}$ 
& 9 & \xmark & \xmark\textsuperscript{*} & \xmark\textsuperscript{*} & \xmark\textsuperscript{*} & \xmark\textsuperscript{*} &
7 & 6 & 5 & 6 & 5 & 4 
\\
\hline
\multirow{3}{*}{ERK4+U4} 
& $2^{8\phantom{1}} \times 2^{10}$ & 
6 & 4 & 8 & \xmark & \xmark\textsuperscript{*} & \xmark\textsuperscript{*} &
5 & 4 & 4 & 4 & 5 & 5
\\
& $2^{10} \times 2^{12}$ & 
6 & 4 & 8 & \xmark & \xmark\textsuperscript{*} & \xmark\textsuperscript{*} &
5 & 4 & 4 & 4 & 5 & 6
\\
& $2^{12} \times 2^{14}$ & 
6 & 4 & 8 & \xmark & \xmark\textsuperscript{*} & \xmark\textsuperscript{*} &
5 & 4 & 4 & 4 & 5 & 6
\\
\hline
\multirow{3}{*}{ERK5+U5} 
& $2^{8\phantom{1}} \times 2^{9\phantom{1}}$ & 
3 & 3 & 7 & \xmark\textsuperscript{*} & \xmark\textsuperscript{*} & \xmark\textsuperscript{*} &
3 & 3 & 3 & 4 & 4 & 3
\\
& $2^{10} \times 2^{11}$ & 
3 & 3 & 7 & \xmark\textsuperscript{*} & \xmark\textsuperscript{*} & \xmark\textsuperscript{*} &
3 & 3 & 3 & 4 & 5 & 4
\\
& $2^{12} \times 2^{13}$ & 
3 & 3 & 7 & \xmark\textsuperscript{*} & \xmark\textsuperscript{*} & \xmark\textsuperscript{*} &
3 & 3 & 3 & 4 & 5 & 4
\\
\hline
\end{tabular}
\end{center}
\end{table}

We now explain in detail how the sparsity patterns were chosen that lead to the results in the right-hand side of Table~\ref{tab:LSQlin_ERK}, and then give a general discussion about the solvers. To select this sparsity pattern for a given discretization and coarsening factor, we first look at the locations of the largest non-zeros in $\Psi_{\rm ideal}$ (as in Figure~\ref{fig:Psi_ideal_entries}, for example). As a first approximation, we choose a contiguous subset of the locations of the largest $\texttt{nnz}(\Phi)$ non-zeros of $\Phi^m$ (even if the locations of the largest non-zeros are not contiguous). The solver is then tested at multiple grid resolutions to determine if it is scalable. If the solver is not scalable, then an extra non-zero is included in a contiguous fashion and it is retested; this process is repeated until a scalable solver is obtained. 

Additionally, once a scalable solver has been found, if it is determined that the convergence is significantly improved by including a relatively small number of additional non-zeros (e.g., two or three), then that is done also. Note, however, that the convergence rate has not been rigorously optimized as a function of the number of non-zeros. As an example, the left panel of Figure~\ref{fig:sparsity_pattern} shows the non-zero patterns of $\Psi$ selected for ERK3+U3 as a function of coarsening factor, $m$. Also plotted is the coarse-grid characteristic departure point of $-mc$, demonstrating how the sparsity patterns are correlated with the departure points. Figure~\ref{fig:sparsity_pattern} also shows (right panel), for each discretization, the operator complexities \eqref{eq:OC} of the resulting solvers along with the operator complexity of $1 + 1/m$ that results when $\texttt{nnz}(\Psi) = \texttt{nnz}(\Phi)$ in a two-level method (see \S \ref{ssec:sparsity_selection}).

We find that, in general, to obtain convergent and scalable solvers there has to be a slight increase in the number of non-zeros in $\Psi$ as the coarsening factor is increased, as can be seen for ERK3+U3 in Figure~\ref{fig:sparsity_pattern} (left panel), for example. This behavior appears consistent with that seen in Yavneh\cite{Yavneh1998} for the multigrid solution of steady-state advection-dominated PDEs, in which it was shown that coarse-grid operators may require a wider stencil than fine-grid operators. The number of additional non-zeros required is smaller for higher-order discretizations, as seen in the right panel of Figure~\ref{fig:sparsity_pattern}, where operator complexities tend to be smaller for higher-order methods. Notice that ERK4+U4 and ERK5+U5 essentially have operator complexities of $1 + 1/m$ since very few (if any) additional non-zeros were needed.

\begin{figure}
\centerline{
\includegraphics[width=0.425\textwidth]{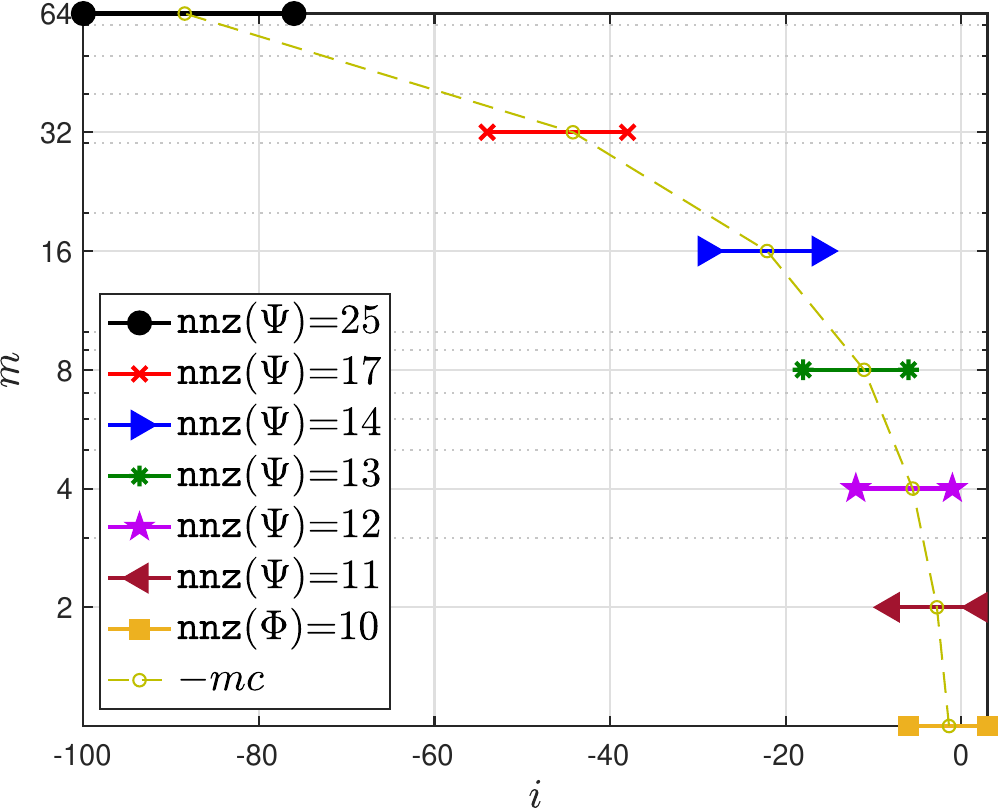}
\quad\quad
\includegraphics[width=0.435\textwidth]{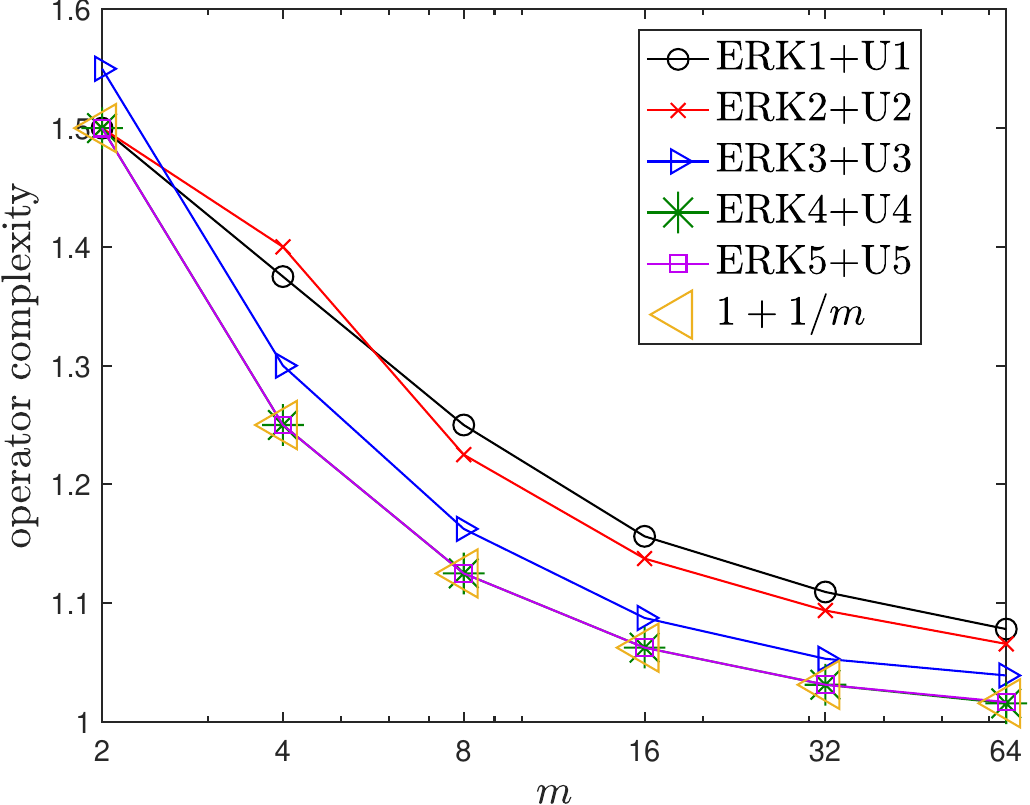}
}
\caption{
Two-level solvers for $m \in \{2,4,8,16,32,64\}$ with $\Psi$ as linear least squares solution \eqref{eq:LSQ_1D_problem}. Left: Sparsity patterns chosen for $\Psi$ for ERK3+U3, as represented by the non-zero diagonal indices $i$ for each value of $m$. Also plotted is $-mc$, which represents the characteristic departure point on a grid using a time-step of $m \Delta t$. Right: Operator complexities \eqref{eq:OC} for all discretizations; shown also is the reference operator complexity of $1 + 1/m$.
\label{fig:sparsity_pattern}
}
\end{figure}

The results at the right of Table~\ref{tab:LSQlin_ERK} show that it is possible to overcome the CFL instability that arises from rediscretizing the fine-grid discretization on a temporally coarsened mesh and to obtain very fast two-level convergence and, therefore, show the significance of using a characteristic-based sparsity pattern for $\Psi$.
Notably, the convergence rates of the solvers shown in Table~\ref{tab:LSQlin_ERK} are comparable to those for model diffusion problems using rediscretized coarse-grid operators.\cite{Dobrev_etal_2017, Falgout_etal_2014} To the best of our knowledge, these are the first scalable results obtained with a two-level time-coarsening algorithm for the explicit discretization of any hyperbolic PDE using realistic CFL numbers, and also for moderately-large coarsening factors. Interestingly, convergence rates tend to be faster for higher-order discretizations compared with those of lower order. When combined with the trend in Figure~\ref{fig:sparsity_pattern} (right panel) that operator complexities are smaller for higher-order discretizations, this suggests that higher-order discretizations of model problem \eqref{eq:PDE} likely benefit more from parallel-in-time integration.

Finally, an example of the eigenvalues and entries of $\Psi$ for ERK3+U3 with $m = 8$ is shown in Figure~\ref{fig:LSQ_Psi_example}. In this example, the eigenvalues of $\Phi^8$ are clearly very well approximated by the eigenvalues of $\Psi$ when they are of order one (in magnitude), and not so well approximated when they are smaller.  
Given this behavior, it is unsurprising that the solver converges quickly, and that the associated error bounds are small (bottom right panel of Figure~\ref{fig:LSQ_Psi_example}). The entries of the least squares $\Psi$ (upper right panel of Figure~\ref{fig:LSQ_Psi_example}) are clearly correlated with those of the ideal operator.

\begin{figure}
\centerline{
\begin{minipage}{0.5\textwidth}
\includegraphics[width=0.85\textwidth]{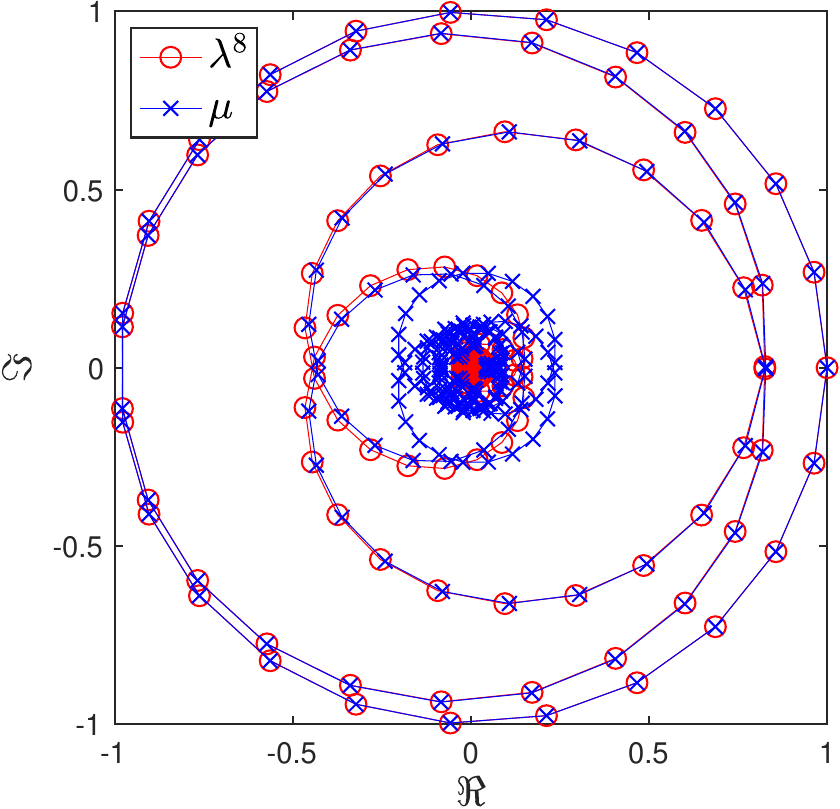}
\end{minipage}
\begin{minipage}{0.5\textwidth}
\vspace{-0ex}
\includegraphics[width=0.85\textwidth]{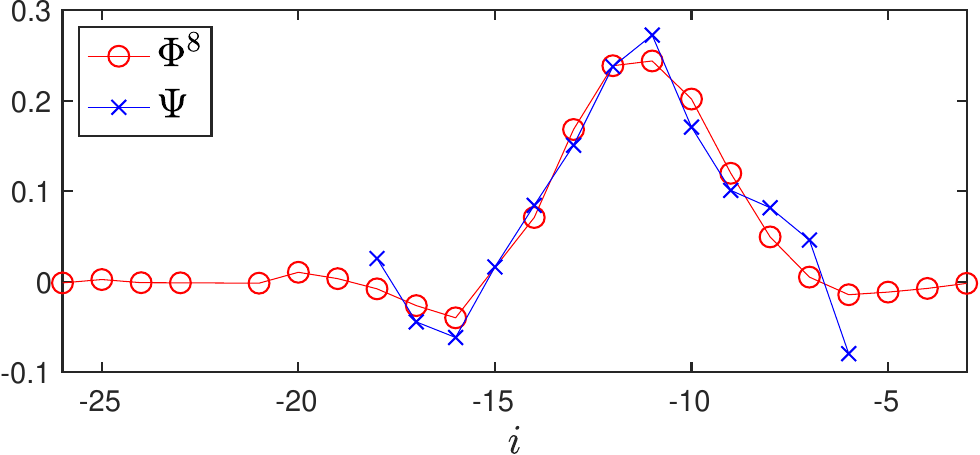}
\\[2ex]
\includegraphics[width=0.85\textwidth]{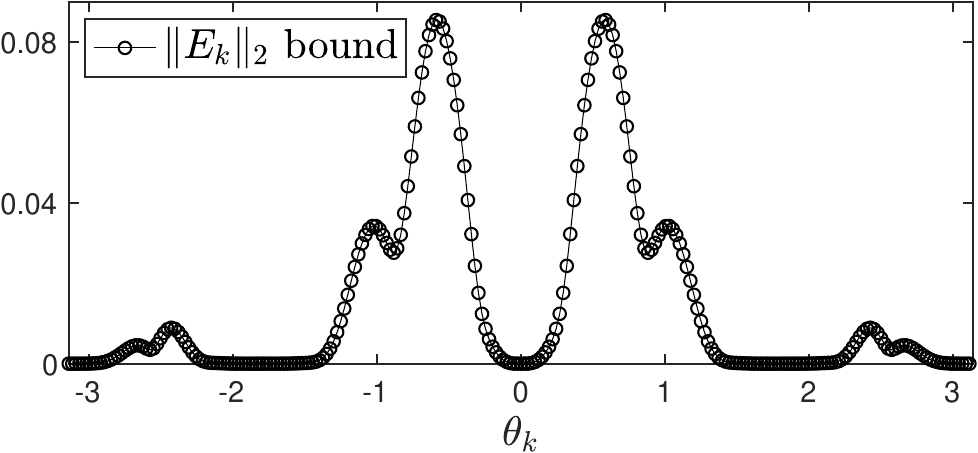}
\end{minipage}
}
\caption{
Linear least squares solution \eqref{eq:LSQ_1D_problem} for ERK3+U3 with $n_x \times n_t = 2^8 \times 2^9$, coarsening factor $m = 8$, and the sparsity pattern of the least-squares determined $\Psi$ based on that of $\Psi_{\rm ideal} = \Phi^8$. Left: Eigenvalues $\lambda^8$ of $\Phi^8$, and $\mu$ of $\Psi$. Top right: Entries of $\Phi^8$ with magnitude larger than $10^{-3}$ as a function of their diagonal index, $i$, and all entries of $\Psi$. Bottom right: Error bound \eqref{eq:Dobrev_bounds} as a function of spatial Fourier frequency. \label{fig:LSQ_Psi_example}
}
\end{figure}

\begin{remark}[Relation between $\Psi$ and semi-Lagrangian discretizations]
It is clear that the optimized coarse operators $\Psi$ with $\Phi^m$-based sparsity have a 
non-local stencil structure that has a very different support than standard Eulerian discretizations
like the ERK+U schemes.
However, the stencil support of these coarse operators $\Psi$ is similar to the stencil support
of semi-Lagrangian schemes, since it tracks the characteristic curves of the PDE.
This begs the question whether the optimized operators $\Psi$ we obtain with $\Phi^m$-based sparsity
imposed may be close to semi-Lagrangian discretizations.
It would indeed be useful for the sake of developing more practical approaches if our
optimized operators $\Psi$ could be replaced by semi-Lagrangian discretizations,
without substantially degrading the convergence speed.
However, since the stencil of $\Psi$ has to be increased with coarsening factor $m$ to get a
scalable solver (see Figure~\ref{fig:sparsity_pattern}), $\Psi$ clearly does not just represent a
particular semi-Lagrangian discretization of the PDE on coarse grids.
This suggests that using standard semi-Lagrangian coarse operators may by itself not be sufficient
to obtain scalable solvers, which is consistent with the findings in Schmitt et al.,\cite{Schmitt_etal_2018}
where fast convergence was not obtained for a hyperbolic problem using a semi-Lagrangian
coarse-grid operator. Nonetheless, exploring semi-Lagrangian coarse-grid operators presents
a promising area for future research.
\end{remark}

\subsection{Explicit schemes: Multilevel results}
\label{ssec:LSQ_multilevel}

In this section, we extend the results from above to define an effective hierarchy of coarse-grid operators for multilevel solvers. The scalability of the two-level solvers considered in the previous section is limited because they require the sequential solve of a large coarse-grid problem. Conversely, multilevel solvers are more scalable because the temporal grid can be coarsened gradually over many levels until the coarsest level contains sufficiently few degrees of freedom that a sequential solve there does not present a significant bottleneck. 

Convergence theory of multilevel MGRIT is significantly more complicated than in the two-level setting\cite{Hessenthaler_etal_2020} and, so, rather than approximately minimizing a multilevel convergence estimate akin to what we did in the two-level case, we simply consider applying our previous two-level strategy in a recursive fashion. That is, if level $\ell$ uses a time-stepping operator $\Phi_{\ell}$, and coarsens by a factor of $m$, then the ideal time-stepping operator on level $\ell+1$, $\Psi_{\textrm{ideal}, \ell+1} := \Phi^m_{\ell}$, is approximated  with linear least squares problem \eqref{eq:LSQ_1D_problem}. As previously, sparsity patterns of coarse-grid operators are selected by roughly choosing some subset of the locations of the largest non-zeros in the corresponding ideal coarse-grid operator. Again, we try to strike some balance between the overall convergence rate of the solver and the amount of fill-in of the coarse-grid operators. 

For the sake of brevity, we only show results for ERK$p$+U$p$, $p \in \{1,3,5\}$. We have considered both V- and F-cycles using coarsening factors of both $m = 2$ and $m = 4$. However, only results for V-cycles using $m = 4$ coarsening are shown here because we found that this combination typically resulted in the fastest parallel solvers  (see \S \ref{sec:parallel}). For the case of ERK1+U1, we coarsen down to a minimum of just four points on the coarsest grid in time, and for ERK3+U3 and ERK5+U5, we coarsen down to a minimum of eight points.

The iteration counts for the resulting solvers are shown in Table \ref{tab:LSQ_multilevel} as a function of mesh resolution and number of grid levels. For all three discretizations, the solvers appear scalable with respect to the number of levels in the grid hierarchy and the mesh size, and they are very fast. We find that to obtain scalable solvers, the number of non-zeros in coarse-grid operators has to increase relative to that of the operator on the previous level. Similarly to the two-level case (see Figure \ref{fig:sparsity_pattern}), the amount of fill-in required decreases with increasing discretization order, as seen by the operator complexities also shown in the table. Importantly, the operator complexities converge to a constant as the number of levels is increased, which, when taken with the scalable iteration counts, indicates that the amount of work to solve a given problem is independent of the number of grid levels. 

To the best of our knowledge, this is the first time that scalable multilevel results have been obtained for the explicit discretization of any hyperbolic PDE using a realistic CFL fraction. For example, Howse et al.\cite{Howse_etal_2019} is the only other work to show multilevel MGRIT results (with spatial coarsening) for explicit discretizations of hyperbolic problems, yet results presented there were limited to first-order accuracy, used a smaller CFL fraction, and even with the use of F-cycles, were not scalable with respect to mesh size. Furthermore, convergence was slow, with on the order of 40 iterations required to reach convergence for the mesh sizes considered here. 

\begin{table}[tbhp] 
\caption{
Multilevel iteration counts as a function of number of grid levels for V-cycles with $m = 4$ coarsening on each level. Operator complexities (OC) \eqref{eq:OC} are also given for each discretization. Note the reference operator complexity for a multilevel method is bounded above by $1 + 1/(m-1) = 1.33\ldots$ for $m = 4$ (see \S \ref{ssec:sparsity_selection}). A `--' denotes a hierarchy that would have coarsened to fewer than the prescribed minimum number of allowable points (four for ERK1+U1, and eight for ERK3+U3 and ERK5+U5).
\label{tab:LSQ_multilevel}
}
\begin{center}
\begin{tabular}{|c|c|cccccc|} 
\hline
\multirow{2}{*}{Scheme} 
&
\multirow{2}{*}{$n_x \times n_t$} 
& 
\multicolumn{6}{c|}{Number of levels} 
\\ 
\cline{3-8}
&
&
2 & 3 & 4 & 5 & 6 & 7 
\\
\Xhline{2\arrayrulewidth} 
\multirow{4}{*}{ERK1+U1} 
& $2^{8\phantom{1}} \times 2^{10}$ & 6 & 6 & 6 & 6 & -- & -- \\ 	
& $2^{10} \times 2^{12}$ & 6 & 7 & 7 & 7 &  7 & -- \\    
& $2^{12} \times 2^{14}$ & 6 & 7 & 7 & 7 &  7 & 7 \\	 
& OC & 1.38 & 1.56 & 1.65 & 1.69 & 1.71 & 1.71 \\
\hline 
\multirow{4}{*}{ERK3+U3} 
& $2^{8\phantom{1}} \times 2^{9\phantom{1}}$ & 6 & 7 & 7 & -- & -- & --  \\ 
& $2^{10} \times 2^{11}$ & 6 & 7 & 7 & 7 & -- & -- \\  
& $2^{12} \times 2^{13}$ & 6 & 7 & 7 & 8 & 8 & -- \\  
& OC & 1.28 & 1.35 & 1.38 & 1.38 & 1.39 & -- \\
\hline
\multirow{4}{*}{ERK5+U5} 
& $2^{8\phantom{1}} \times 2^{9\phantom{1}}$ & 3 & 4 & 4 & -- & --	& --	\\	
& $2^{10} \times 2^{11}$ & 3 & 4 & 5 & 5 & -- & --	\\	
& $2^{12} \times 2^{13}$ & 3 & 4 & 5 & 5 & 5 & --	\\	
& OC & 1.25 & 1.31 & 1.33 & 1.33 & 1.33 & -- \\
\hline
\end{tabular}
\end{center}
\end{table}

\subsection{Explicit schemes: Application to inflow/outflow boundaries}
\label{ssec:LSQ_inflow}
Discretizations of \eqref{eq:PDE} subject to inflow/outflow spatial boundary conditions  result in non-circulant $\Phi$. Therefore, the optimization techniques discussed in previous sections cannot be rigorously applied since they rely on $\Phi$ and $\Psi$ being circulant. Moreover, such $\Phi$ are non-normal and likely not even diagonalizable, and, so, the rigorous optimization of the corresponding $\Psi$ would require more sophisticated convergence theory than in both Dobrev et al.\cite{Dobrev_etal_2017} and Hessenthaler et al.,\cite{Hessenthaler_etal_2020} such as that of Southworth\cite{Southworth2019}, for example, and would certainly be highly nonlinear.

In the spirit of local Fourier analysis of multigrid methods,\cite{Trottenberg_etal_2001} we hypothesize that the $\Psi$ we have previously designed for the periodic problem may work well for the inflow/outflow problem since the fine-grid operators $\Phi$ share the same Toeplitz structure away from the boundaries. To test this, we apply MGRIT to inflow/outflow problems using coarse-grid operators as described in the previous sections that were designed for analogous periodic problems (i.e., same discretizations, CFL numbers, $n_x$, and $m$), and we truncate these at the boundaries such that they are no longer circulant but remain Toeplitz. Note that not truncating the operators leads to similar, but slightly less satisfactory results. Given their non-zero pattern, truncating these operators at boundaries results in strictly lower triangular matrices in almost all cases.

In our tests, an inflow boundary condition $u(-1,t) = \sin^4(\pi t)$ is prescribed at $x = -1$. At the outflow boundary $x = 1$, no boundary condition is specified since the solution simply propagates out of the domain along characteristics. To implement the outflow numerically, however, we extrapolate from inside the computational domain to ghost points since solution data at ghost points is required by the 3rd-order and higher spatial discretizations. To preserve the global $p$th-order accuracy of each spatial discretization, we employ $p$th-order extrapolation.
While the inflow condition leads to a solution that mimics the periodic solution (they converge to the same solution as the mesh is refined), this does not influence the convergence of MGRIT, which is independent of the solution for linear problems.\cite{Southworth2019} For the numerical implementation of boundary conditions, sufficiently accurate extrapolation is used at the outflow boundary; at the inflow boundary, sufficiently accurate ERK stage values are computed using ideas similar to those in Carpenter et al.,\cite{Carpenter_etal_1993} except we elect to use the same spatial discretization right up to the boundary rather than switching to a compact one. To approximate solution and ERK stage values at ghost points, we employ truncated Taylor series about the boundary and use the PDE with the `inverse Lax--Wendroff' procedure.\cite[p. 364]{Hesthaven2017} For each scheme, numerical tests (not shown here) verify that convergence at the theoretically predicted rates is achieved. Tests also indicate that CFL limits are very similar to their analogues with periodic boundaries (Table~\ref{tab:CFL_limits}).

Iteration counts for the inflow/outflow problem discretized with ERK$p$+U$p$, $p \in \{1,3,5\}$, are given in Table~\ref{tab:LSQlin_ERK_inflow}. Two-level solvers using different coarsening factors $m$, and multilevel V-cycles using $m = 4$ on each level are considered. On average, the results are very similar to those for the analogous periodic problems (right side of Table~\ref{tab:LSQlin_ERK} for two level, and Table~\ref{tab:LSQ_multilevel} for multilevel). The optimized coarse-grid operators for the periodic problem therefore also make excellent coarse-grid operators for the inflow/outflow problem despite them not being designed to do so in any rigorous sense. These results indicate that the issues hindering convergence for hyperbolic problems in the simpler periodic setting, where $\Phi$ and $\Psi$ are normal matrices, are also responsible for poor convergence in this more complicated setting. 
\begin{table}[tbhp] 
\caption{
Iteration counts for ERK+U discretizations of \eqref{eq:PDE} with inflow/outflow boundaries; $\Psi$ is given by truncating the circulant matrix resulting from linear least squares solution \eqref{eq:LSQ_1D_problem}, with its sparsity pattern based on that of $\Psi_{\rm ideal}$. Left: Two-level solves as a function of coarsening factor $m$. Right: $\ell$-level V-cycle solves using $m = 4$ coarsening on each level. A `--' denotes a hierarchy that would have coarsened to fewer than the prescribed minimum number of allowable points (four for ERK1+U1, and eight for ERK3+U3 and ERK5+U5).
\label{tab:LSQlin_ERK_inflow}
}
\begin{center}
\begin{tabular}{|c|c|cccccc|cccccc|} 
\hline
\multirow{2}{*}{Scheme} 
&
\multirow{2}{*}{$n_x \times n_t$} 
& 
\multicolumn{6}{c|}{Two level, $m$} 
&
\multicolumn{6}{c|}{Multilevel, $\ell$} \\ \cline{3-14}
&
& 
2 & 4 & 8 & 16 & 32 & 64 
&
2 & 3 & 4 & 5 & 6 & 7 
\\ 
\Xhline{2\arrayrulewidth}
\multirow{3}{*}{ERK1+U1} &
$2^{8\phantom{1}} \times 2^{10}$ & 
10 & 6 & 6 & 6 & 5 & 3 
&
6 & 6 & 6 & 6 & -- & -- 
\\
&
$2^{10} \times 2^{12}$ & 
11 & 6 & 6 & 7 & 6 & 5 
&
6 & 6 & 6 & 6 & 6 & --  
\\ 
&
$2^{12} \times 2^{14}$ & 
11 & 6 & 6 & 7 & 6 & 5 
&
6 & 7 & 7 & 7 & 7 & 7  
\\ 
\hline
\multirow{3}{*}{ERK3+U3} &
$2^{8\phantom{1}} \times 2^{9\phantom{1}}$ & 
7 & 6 & 5 & 5 & 4 & 2 
&
6 & 6 & 6 & -- & -- & -- 
\\ 
&
$2^{10} \times 2^{11}$ & 
7 & 6 & 5 & 6 & 5 & 4 
&
6 & 7 & 7 & 7 & -- & -- 
\\ 
&
$2^{12} \times 2^{13}$ & 
7 & 6 & 5 & 6 & 5 & 4 
&
6 & 7 & 7 & 7 & 7 & -- 
\\
\hline
\multirow{3}{*}{ERK5+U5} &
$2^{8\phantom{1}} \times 2^{9\phantom{1}}$ & 
8 & 5 & 4 & 4 & 4 & 2 
&
5 & 5 & 5 & -- & -- & -- 
\\ 
&
$2^{10} \times 2^{11}$ & 
7 & 5 & 3 & 4 & 4 & 4 
&
5 & 5 & 5 & 5 & -- & -- 
\\
&
$2^{12} \times 2^{13}$ & 
7 & 5 & 3 & 4 & 4 & 4 
&
5 & 5 & 5 & 5 & 5 & -- 
\\ 
\hline
\end{tabular}
\end{center}
\end{table}

\subsection{Implicit schemes}
\label{ssec:implicit}
We now consider linear least squares problem \eqref{eq:LSQ_1D_problem} for the computation of suitable coarse-grid operators $\Psi$ for SDIRK+U discretizations of \eqref{eq:PDE}. For such discretizations, $\Phi$ is a rational function of sparse matrices and so, too, is $\Psi_{\rm ideal} := \Phi^m$. Naturally, one might seek a $\Psi$ that is also of this form. However, it is not obvious how this should be done, with one complication being the choice of sparsity patterns for the numerator and denominator. Consequently, we take a different approach here. 

Since $\Phi$ is a rational function of sparse matrices, it can also be written as a dense matrix. To assess to what extent $\Phi$ and $\Phi^m$ do globally couple the solution, we consider the magnitude of their entries as a function of their diagonal index, as pictured in Figure~\ref{fig:Psi_ideal_entries_SDIRK} for $m \in \{ 16,64 \}$. These plots show that, despite $\Phi$ and $\Phi^m$ being dense, they effectively act as sparse matrices, with their largest non-zeros having a sharp peak that is correlated with the characteristic departure point (shown as the dashed vertical line). The one exception here is SDIRK1+U1, whose entries are significantly less peaked than the other discretizations. As in previous examples, this is consistent with this discretization being very dissipative and not capturing the non-dissipative nature of \eqref{eq:PDE} well. Indeed, the plots in Figure~\ref{fig:Psi_ideal_entries_SDIRK} are qualitatively similar to their analogues for the ERK schemes in Figure~\ref{fig:Psi_ideal_entries} (noting the curves sit over the top of one another in the SDIRK case because the same CFL number of $c=4$ is used for every implicit discretization).

\begin{figure}[h!t]
\centerline{
\includegraphics[width = 0.44\textwidth]{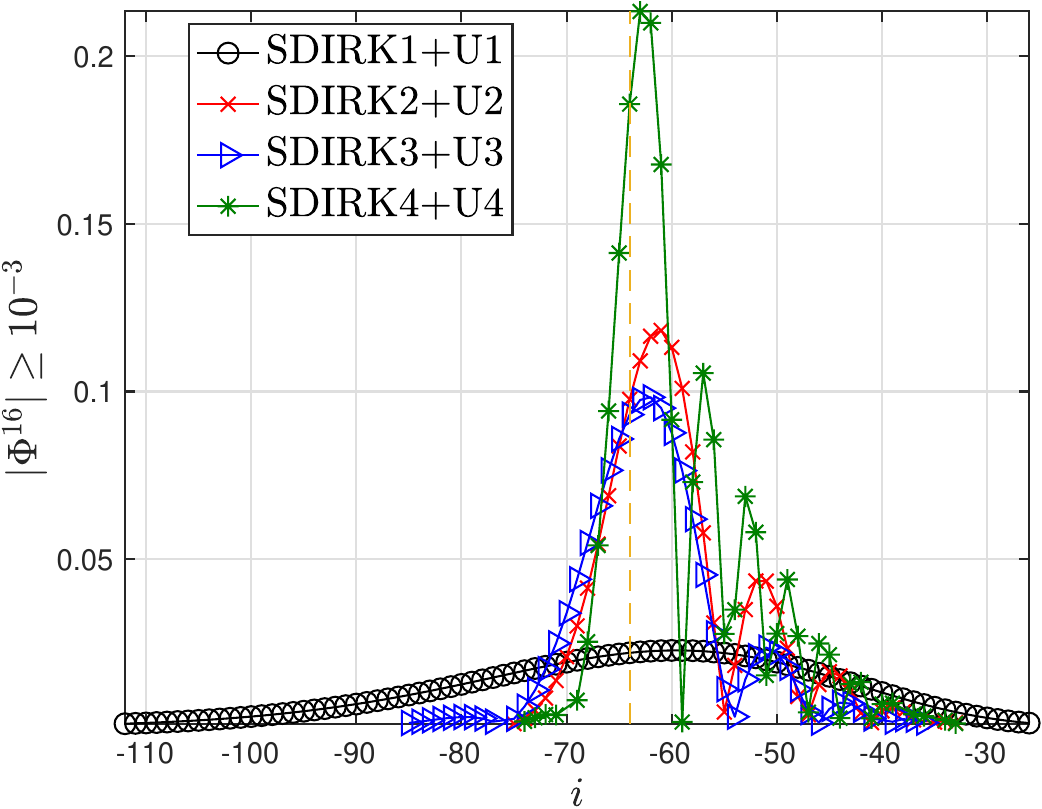}
\quad\quad
\includegraphics[width = 0.44\textwidth]{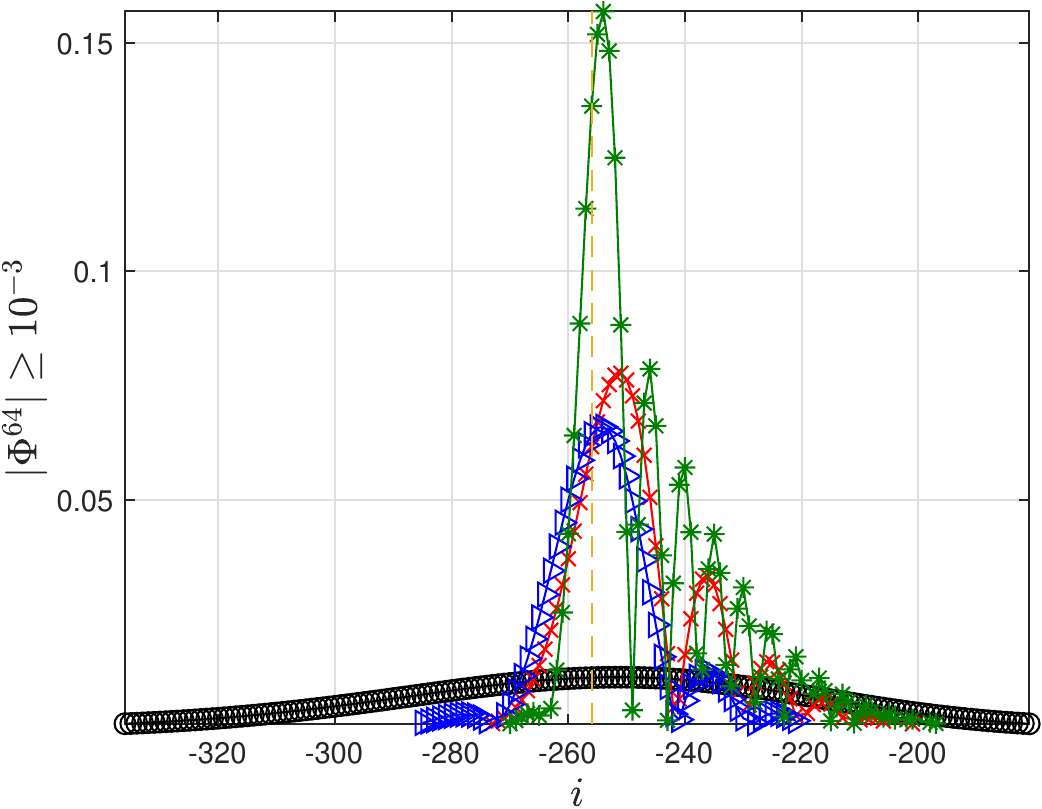}
}
\caption{
Magnitude of diagonal entries of the dense matrices $\Psi_{\rm ideal} := \Phi^m$, $m=16$ (left), and $m=64$ (right), that are larger than $10^{-3}$, as a function of their diagonal index, $i$. Fine-level discretizations, $\Phi$, are SDIRK$p$+U$p$, $p \in \{1,2,3,4\}$. A value of $n_x = 2^{10}$ has been used here. In each plot, a dashed vertical line is included to represent a distance of $-4 m \Delta x$ from $i = 0$ (these schemes use CFL number of $c = 4$). 
\label{fig:Psi_ideal_entries_SDIRK}
}
\end{figure}

The effectively sparse structure of $\Psi_{\rm ideal}$---as shown in Figure~\ref{fig:Psi_ideal_entries_SDIRK}---begs the question: Can a sparse (or equivalently, explicit) coarse-grid operator $\Psi$ be used to approximate it? The use of an explicit coarse-grid operator with an implicit fine-grid discretization is certainly not standard, and in fact, the reverse case has been used elsewhere in the literature: An implicit coarse-grid operator has been coupled with an explicit fine-grid discretization since it is a natural way of ensuring that the coarse-grid operator is stable.\cite{Falgout_etal_2014} However, quasi-tracking the solution of the PDE along characteristics---as done in the previous sections---is another way of ensuring the coarse-grid operator is stable, since the physical domain of dependence is included in the numerical domain of dependence. 

We now test the idea of using a sparse $\Psi$ to approximate a dense $\Phi^m$. As for the ERK discretizations, we place a restriction on the number of non-zeros in $\Psi$. To do so, we compute the entries in the 1st column of $\Phi^m$ (this can be done using the FFT and its inverse), and then we select a non-zero pattern using thresholding. That is, recalling $\tilde{\bm{\phi}}^m$ is the (dense) first column of $\Phi^m$, we take the non-zero pattern to be that of the entries with magnitude at least equal to $\eta_{\rm tol} \times \max_{k} |\tilde{\phi}^m_{k} |$, in which $\eta_{\rm tol} < 1$. We find that smaller values of $\eta_{\rm tol}$ lead to more quickly converging MGRIT solvers. As for the ERK schemes, we have loosely tried to achieve some balance between the rate of convergence and the number of non-zeros in $\Psi$, but this has not been fully optimized. For each discretization and coarsening factor, $m$, we allow for a different value of $\eta_{\rm tol}$. For $m = (2,4,8,16,32,64)$ the values for the $p$th-order SDIRK+U scheme are: $p = 1$, $\eta_{\rm tol} = (.1,.125,.25,.5,.5,.6)$;  $p = 2$,  $\eta_{\rm tol} = (.05,.1,.1,.2,.2,.2)$; $p = 3$, $\eta_{\rm tol} = (.005,.01,.02,.02,.02,.04)$; and $p = 4$, $\eta_{\rm tol} = (.005,.01,.01,.01,.02,.02)$. These choices of $\eta_{\rm tol}$ result in coarse-grid operators that have on the order of the same number of entries shown in the plots in Figure~\ref{fig:Psi_ideal_entries_SDIRK}. Figure~\ref{fig:nPsi_diags_SDIRK} shows the number of non-zeros per row of $\Psi$ as a function of the coarsening factor and how there is, in general, some growth in this number with $m$, just as there is in the number of non-zeros in $\Phi^m$ whose magnitude is significant (Figure~\ref{fig:Psi_ideal_entries_SDIRK}).
\begin{figure}[h!t]
\centerline{\includegraphics[width=0.45\textwidth]{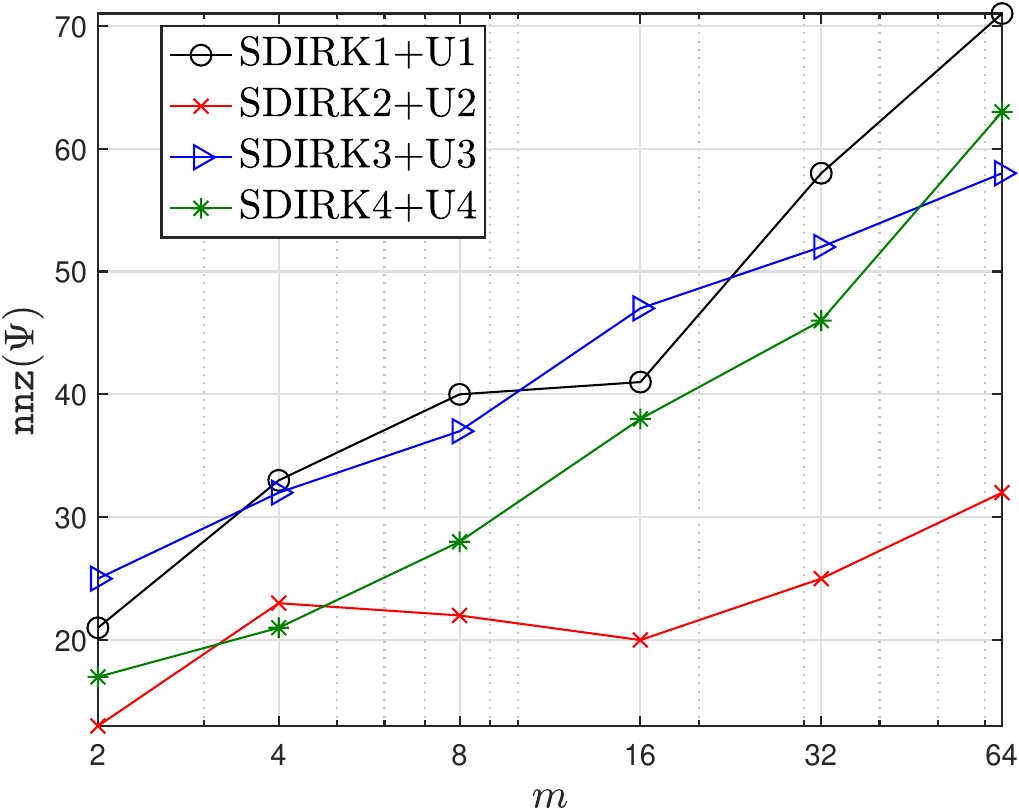}}
\caption{
Number of non-zeros per row of $\Psi$ for SDIRK+U discretizations as a function of coarsening factor, $m$. 
\label{fig:nPsi_diags_SDIRK}
}
\end{figure}

The iteration counts for the solvers are shown in Table~\ref{tab:LSQlin_SDIRK}. Convergence is fast for all coarsening factors, and the solvers appear scalable as the mesh is refined. This is in stark contrast to the results in Table~\ref{tab:examples_IRK_redisc} where rediscretizing $\Phi$ on the coarse grid resulted in a divergent solver for all discretizations except for SDIRK1+U1, reinforcing the idea that there exist significantly better coarse-grid operators for advection problem \eqref{eq:PDE} than those offered by rediscretizing the PDE on the coarse grid. Furthermore, these results confirm that despite $\Phi^m$ being a dense operator for the implicit temporal discretizations considered here, it can be well approximated by a sparse one.

\begin{table}[tbhp] 
\caption{
Two-level iteration counts for SDIRK+U discretizations with $\Psi$ given as linear least squares solution \eqref{eq:LSQ_1D_problem}. 
\label{tab:LSQlin_SDIRK}
} 
\begin{center}
\begin{tabular}{|c|c|cccccc|} 
\hline
\multirow{2}{*}{Scheme} 
&
\multirow{2}{*}{$n_x \times n_t$} 
& 
\multicolumn{6}{c|}{$m$} \\ \cline{3-8}
&
& 
2 & 4 & 8 & 16 & 32 & 64
\\ \Xhline{2\arrayrulewidth}
\multirow{2}{*}{SDIRK1+U1} 
& $2^{10} \times 2^{10}$ & 10 & 7 & 8 & 10 & 8 & 7\\
& $2^{12} \times 2^{12}$ & 10 & 7 & 8 & 11 & 9 &  9\\
\hline
\multirow{2}{*}{SDIRK2+U2} 
& $2^{10} \times 2^{10}$ & 10 & 8 & 7 & 8 & 8 & 7\\ 
& $2^{12} \times 2^{12}$ & 11 & 8  & 7 & 8 & 8 & 8\\ 
\hline
\multirow{2}{*}{SDIRK3+U3} 
& $2^{10} \times 2^{10}$ & 5 & 5 & 5 & 4 & 4 & 4 \\ 
& $2^{12} \times 2^{12}$ & 5 & 5 & 5 & 4 & 4 & 4 \\ 
\hline
\multirow{2}{*}{SDIRK4+U4} 
& $2^{10} \times 2^{10}$ & 6 & 6 & 5 & 5 & 5 & 5 \\ 
& $2^{12} \times 2^{12}$ & 6 & 6 & 5 & 5 & 5 & 5 \\ 
\hline
\end{tabular}
\end{center}
\end{table}

\section{Parallel results}
\label{sec:parallel}
In this section, we present strong parallel scaling results for the ERK$p$+U$p$, with $p \in \{1,3,5\}$, multilevel solvers developed in \S \ref{ssec:LSQ_multilevel}. We show that they can lead to significant speed-ups over sequential time-stepping. We have obtained analogous parallel results for the two-level solvers developed in \S \ref{ssec:LSQ_results}, but these lead to smaller speed-ups and so have been omitted here.

The implementations use the open-source package XBraid.\cite{xbraid-package} The results were generated on Quartz, a Linux cluster at Lawrence Livermore National Laboratory consisting of 2{,}688 compute nodes, with dual 18-core 2.1 GHz Intel Xeon processors per node. For each discretization, we consider the strong scaling of a single problem whose space-time grid is the largest from Table~\ref{tab:LSQ_multilevel}, and the number of levels in the solver is taken as the maximum shown in this table. Since we want to demonstrate the benefits of parallelization in time, we only consider parallelization in the time direction. As throughout the rest of this paper, we first consider a stopping criterion based on achieving a space-time residual below $10^{-10}$ in the discrete $L_2$-norm, but a stopping criterion based on achieving discretization error accuracy is also considered. 

In our parallel tests, we have considered both V- and F-cycles with coarsening factors of $m = 2$ and $m = 4$. We find that F-cycles require fewer iterations to converge than V-cycles, but this of course comes at the cost of added work and communication. Accordingly, we typically find that the best results arise from the use of V-cycles with $m = 4$ coarsening and, thus, results for this configuration are shown here, in Figure~\ref{fig:strong_scaling_Vcycle}. The plots show good parallel scaling with benefit over sequential time-stepping when using at least 32 processors in almost all cases, which is on par with what has been achieved for model diffusion-dominated problems using time-only parallelism.\cite{Falgout_etal_2017} The largest speed-up achieved over sequential time-stepping is at 1024 processors, where MGRIT is faster by a factor of about 3.8, 8.4, and 18.1 when solving up to $10^{-10}$ residual tolerance, and of about 10.0, 12.6, and 13.7 when solving up to discretization error (for the discretizations in the order of increasing accuracy). 

The relative speed-ups shown here further demonstrate the improvements given by this work over existing parallel-in-time strategies for hyperbolic PDEs. For example, achieving MGRIT speed-up with high-order discretizations of any hyperbolic problem is unheard of in the literature, and so the fact that we have been able to achieve a speed-up on the order of 15 times for a highly accurate explicit 5th-order discretization run at a realistic CFL fraction is significant.

\begin{figure}[h!t]
\centerline{
\includegraphics[width=0.448\textwidth]{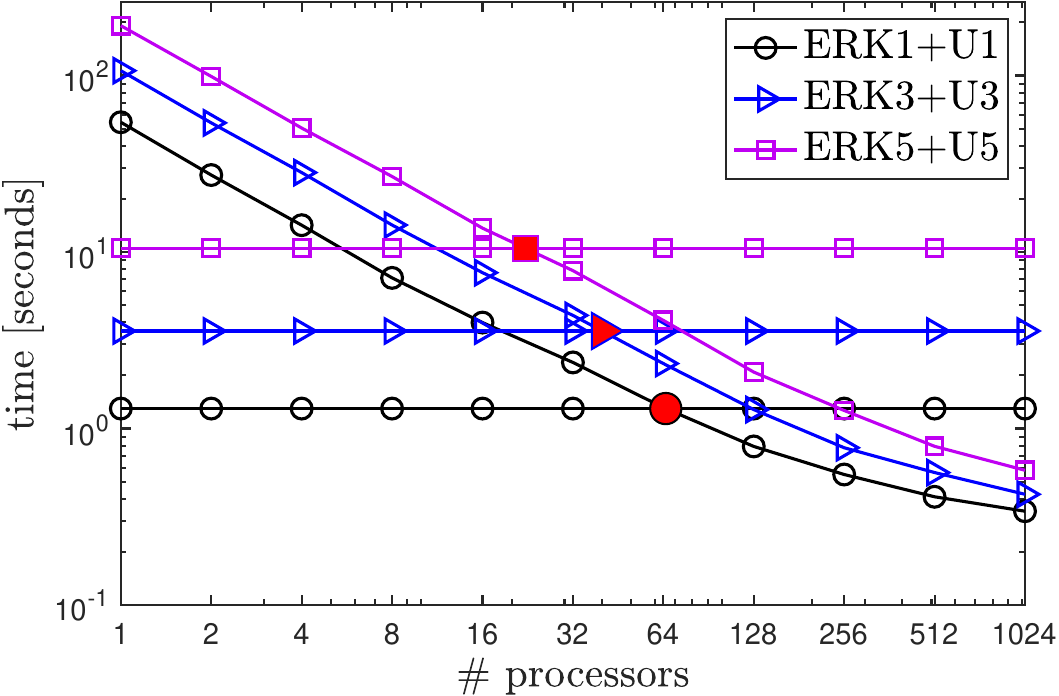}
\quad\quad
\includegraphics[width=0.448\textwidth]{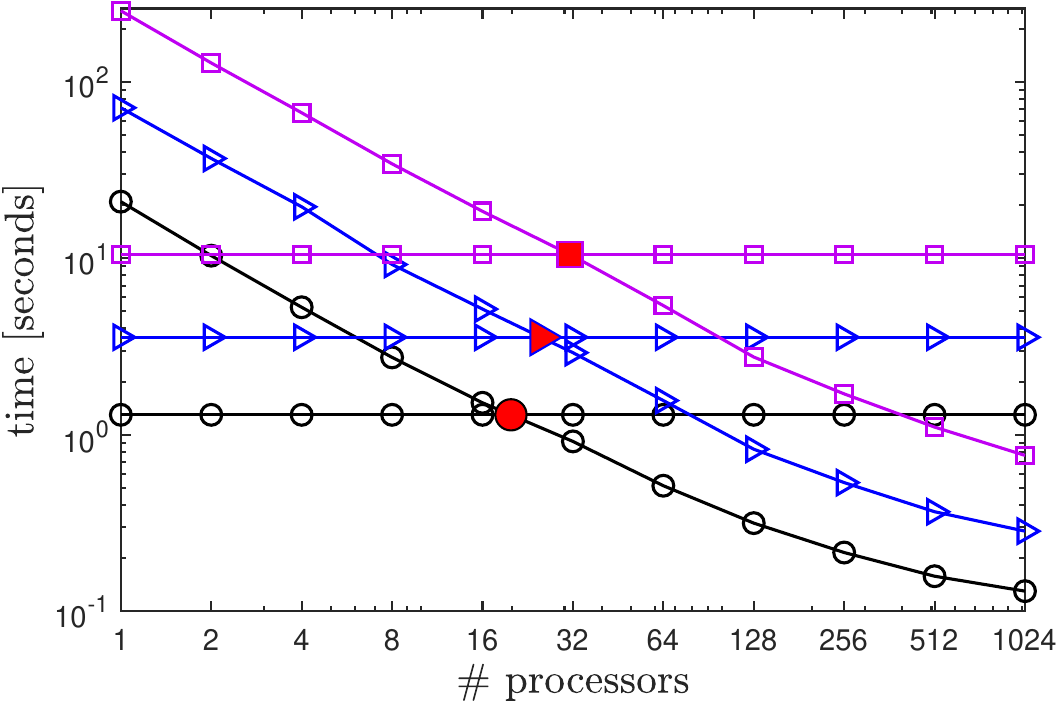}
}
\caption{
Strong parallel scaling: Runtimes of MGRIT V-cycles with $m=4$ coarsening and using time-only parallelism for ERK$p$+U$p$ discretizations on space-time grids of size  $n_x \times n_t = 2^{12}\times (2^{14},2^{13},2^{13})$ for $p = (1,3,5)$. Left: Fixed residual stopping tolerance of $10^{-10}$. Right: Residual stopping tolerance based on the discretization error. Dashed lines represent runtimes of time-stepping on one processor for reference purposes. Solid red markers represent crossover points. 
\label{fig:strong_scaling_Vcycle}
}
\end{figure}

\section{Conclusions}
\label{sec:conclusions}
In this paper, we consider the parallel-in-time integration of the one-dimensional, constant-coefficient linear advection problem using the MGRIT and Parareal algorithms. This PDE represents the simplest of all hyperbolic problems, yet, to the best of our knowledge, no parallel-in-time solvers have been successfully applied to accurate discretizations of this problem, yielding inexpensive solvers that achieve fast and scalable convergence for realistic CFL fractions close to one. We use existing convergence theory to explain why this problem is so difficult, and what is required of coarse-grid operators for its efficient solution. Convergence hinges on the coarse-grid operator accurately propagating slowly-decaying spatial modes through time very similarly to the fine-grid operator. The larger number of such modes for advection-dominated problems compared with diffusion-dominated problems means that even small differences between fine- and coarse-grid operators typically result in extremely poor convergence.

For this PDE, we develop `near-optimal' coarse-grid operators through the approximate minimization of two-level error estimates. We use these coarse-grid operators for both explicit and implicit discretizations of low- and high-order accuracy and demonstrate that they lead to solvers with fast and scalable convergence that is on par with performance typically seen from these parallel-in-time algorithms when applied to diffusion-dominated problems. For explicit discretizations, we show that it is possible to overcome the CFL-driven divergence that arises from naively applying a conditionally stable discretization on the coarse grid. Primarily, this is done by tracking information along characteristic curves of the PDE on the coarse grid. Moreover, we show that observing this characteristic nature is also important for unconditionally stable implicit discretizations. Finally, our results indicate that higher-order discretizations of this problem benefit more from parallel-in-time integration compared with those of lower order. 
 
While the optimization approach proposed here is very effective, it also has clear limitations, requiring the solution of expensive optimization problems. Further, it is not directly generalizable to more complicated PDEs than constant-coefficient linear advection. However, we expect that the heuristic insights developed here will be important for more complicated hyperbolic problems. Future work will focus on developing practical coarse-grid operators for linear advection as well as for more complicated hyperbolic problems which are of practical interest, such as those with non-constant coefficients and nonlinearities. Our major finding, that coarse operators need to track characteristics, will be a central guiding principle in this future work; indeed, this has already led to some promising results in initial research in these directions. 
Furthermore, the spatial parallelization of the coarse-grid operators developed here will also be considered, which is not straightforward due to their coupling of connections over longer distances in space.
 
\appendix
\section{Runge-Kutta Butcher tableaux}
\label{sec:butcher}

For completeness, Butcher tableaux are given here for the Runge-Kutta methods used in the main text. Explicit Runge-Kutta (ERK) methods are given in Table~\ref{tab:ERK1_2_3}, and L-stable singly diagonally implicit Runge-Kutta (SDIRK) methods are given in Table~\ref{tab:SDRK}.

\begin{center}
\begin{table*}[h!t]%
\caption{
Butcher tableaux for ERK methods of orders 1--5.
\label{tab:ERK1_2_3}
}
\centering
\begin{tabular*}{\textwidth}{@{\extracolsep\fill}ccccc@{\extracolsep\fill}}
	ERK1 
	& 
	ERK2 \cite[(9.7)]{Hesthaven2017} 
	& 
	ERK3 \cite[(9.8)]{Hesthaven2017}
	& 
	ERK4 \cite[p. 180]{Butcher2003} 
	& 
	ERK5 \cite[(236a)]{Butcher2003}\\
	\begin{tikzpicture}
	\node (0,0) {$\renewcommand\arraystretch{1.6}\arraycolsep=6pt
	\begin{array}
	{c|c}
	0 & 0\\
	\cline{1-2}
	& 1
	\end{array}$};
	\end{tikzpicture} 
	&	
	\begin{tikzpicture}
	\node (0,0) {$\renewcommand\arraystretch{1.6}\arraycolsep=6pt
	\begin{array}
	{c|cc}
	0	& 0 					& 0 \\
	1	& 1 					& 0 \\
	\cline{1-3}
	 	& \tfrac{1}{2}	& \tfrac{1}{2}
	\end{array}$};
	\end{tikzpicture} 
	&	
	\begin{tikzpicture}
	\node (0,0) {$\renewcommand\arraystretch{1.6}\arraycolsep=6pt
	\begin{array}
	{c|ccc}
	0 						& 0 					& 0 					& 0 	\\
	1 						& 1 					& 0 					& 0	\\
	\tfrac{1}{2}		& \tfrac{1}{4} 	& \tfrac{1}{4}	& 0	\\
	\cline{1-4}
							& \tfrac{1}{6}	& \tfrac{1}{6}	& \tfrac{2}{3}
	\end{array}$};
	\end{tikzpicture} 
	&
	\begin{tikzpicture}
	\node (0,0) {$\renewcommand\arraystretch{1.6}\arraycolsep=6pt
	\begin{array}
	{c|cccc}
         0         			&          0         	&           0        	& 0 					& 0\\
		\tfrac{1}{2}	& \tfrac{1}{2} 	&           0        	& 0 					& 0\\
		\tfrac{1}{2}	&          0         	& \tfrac{1}{2} 	& 0 					& 0\\
            1       			&          0         	&           0         	& 1 					& 0\\
		\cline{1-5}
							& \tfrac{1}{6} 	& \tfrac{1}{3}	& \tfrac{1}{3}	& \tfrac{1}{6} 
	\end{array}$};
	\end{tikzpicture} 
	&	
	\begin{tikzpicture}
	\node (0,0) {$\renewcommand\arraystretch{1.6}\arraycolsep=6pt
	\begin{array}
	{c|cccccc}
         0         		&          0          	&           0         	&         0           		&         0           		& 0 						& 0\\
	\tfrac{1}{4}  	& \tfrac{1}{4}  	&           0         	&         0           		&         0           		& 0 						& 0\\
	\tfrac{1}{4}  	& \tfrac{1}{8}  	&   \tfrac{1}{8}	&         0           		&         0           		& 0 						& 0\\
	\tfrac{1}{2}  	&          0          	&           0         	& \tfrac{1}{2}  		&         0             	& 0 						& 0\\
	\tfrac{3}{4}	& \tfrac{3}{16}	& -\tfrac{3}{8}	& \tfrac{3}{8} 		& \tfrac{9}{16}  	& 0 						& 0 \\
          1          	& -\tfrac{3}{7} 	&   \tfrac{8}{7}	& \tfrac{6}{7} 		& -\tfrac{12}{7} 	& \tfrac{8}{7} 		& 0\\
	\cline{1-7}
						& \tfrac{7}{90} & 0 					& \tfrac{32}{90} 	& \tfrac{12}{90}	& \tfrac{32}{90}	& \tfrac{7}{90} 
	\end{array}$};
	\end{tikzpicture}
\end{tabular*}
\end{table*}
\end{center}

\begin{center}
\begin{table*}[h!t]%
\caption{
Butcher tableaux for L-stable SDIRK methods of orders 1--4. Constants used in SDIRK3 are: 
$\zeta~=~0.43586652150845899942\ldots, \,
\alpha = \tfrac{1+\zeta}{2}, \,
\beta = \tfrac{1-\zeta}{2}, \,
\gamma = -\tfrac{3}{2}\zeta^2 + 4\zeta - \tfrac{1}{4}, \,
\epsilon = \tfrac{3}{2}\zeta^2 - 5\zeta + \tfrac{5}{4}$.
\label{tab:SDRK}
}
\centering
\begin{tabular*}{\textwidth}{@{\extracolsep\fill}cccc@{\extracolsep\fill}}
	SDIRK1 
	& 
	SDIRK2 \cite[p. 261]{Butcher2003} 
	& 
	SDIRK3 \cite[p. 262]{Butcher2003} 
	& SDIRK4 \cite[(6.16)]{Hairer_Wanner1996} 
	\\
	\begin{tikzpicture}
	\node (0,0) {$\renewcommand\arraystretch{1.6}\arraycolsep=6pt
	\begin{array}
	{c|c}
	1 & 1\\
	\cline{1-2}
	& 1
	\end{array}$};
	\end{tikzpicture} 
	&	
	\begin{tikzpicture}
	\node (0,0) {$\renewcommand\arraystretch{1.6}\arraycolsep=6pt
	\begin{array}
	{c|cc}
	1 - \tfrac{\sqrt{2}}{2}	& 1 - \tfrac{\sqrt{2}}{2}	& 0 \\
	1 								    	& \tfrac{\sqrt{2}}{2}      	& 1 - \tfrac{\sqrt{2}}{2} \\
	\cline{1-3}
										& \tfrac{\sqrt{2}}{2} 		& 1 - \tfrac{\sqrt{2}}{2}	
	\end{array}$};
	\end{tikzpicture} 
	&	
	\begin{tikzpicture}
	\node (0,0) {$\renewcommand\arraystretch{1.6}\arraycolsep=6pt
	\begin{array}
	{c|ccc}
	\zeta	& \zeta   		& 0  				& 0  \\
	\alpha  	& \beta  		& \zeta  		& 0  \\
	1   		& \gamma	& \epsilon	& \zeta  \\
	\cline{1-4}
				& \gamma 	& \epsilon 	& \zeta
	\end{array}$};
	\end{tikzpicture} 
	& 
	\begin{tikzpicture}
	\node (0,0) {$\renewcommand\arraystretch{1.6}\arraycolsep=6pt
	\begin{array}
	{c|ccccc}
	\tfrac{1}{4}      	& \tfrac{1}{4}            & 0                                	& 0                        	& 0                        	& 0 \\
	\tfrac{3}{4}     	& \tfrac{1}{2}           	& \tfrac{1}{4}               	& 0                        	& 0                        	& 0 \\
	\tfrac{11}{20}	& \tfrac{17}{50}       	& -\tfrac{1}{25}           	& \tfrac{1}{4}      	& 0                        	& 0 \\
	\tfrac{1}{2}      	& \tfrac{371}{1360}	& -\tfrac{137}{2720}	& \tfrac{15}{544}	& \tfrac{1}{4}     	& 0 \\
	1                       	& \tfrac{25}{24}      	& -\tfrac{49}{48}      	& \tfrac{125}{16} 	& -\tfrac{85}{12}	& \tfrac{1}{4} \\
	\cline{1-6}
							& \tfrac{25}{24}     	& -\tfrac{49}{48}        	& \tfrac{125}{16} 	& -\tfrac{85}{12}	& \tfrac{1}{4}
	\end{array}$};
	\end{tikzpicture} 
\end{tabular*}
\end{table*}
\end{center}
\renewcommand*{\arraystretch}{1} 

\section*{Acknowledgments}
This work was performed under the auspices of the U.S. Department of Energy by Lawrence Livermore National Laboratory under Contract DE-AC52-07NA27344 (LLNL-JRNL-789101).
The work of O.K. was supported by an Australian Government Research Training Program Scholarship, and was partially supported by International Mobility Programme funding from the ARC Centre of Excellence for Mathematical and Statistical Frontiers.  The work of S.M. and H.D.S. was partially supported by NSERC discovery grants RGPIN-2014-06032, RGPIN-2019-05692 and RGPIN-2019-04155.
The authors are grateful to Matt Menickelly for helpful conversations regarding the optimization aspects of this work, and would like to thank Irad Yavneh for his insightful comments on connections between this problem and that considered in Yavneh.\cite{Yavneh1998}
The authors are grateful to the anonymous reviewers for their constructive feedback, which has significantly improved this article.
This study does not have any conflicts of interest.

\nocite{*}
\bibliography{mgrit_cg_selection_refs}%

\vspace{1ex}
\tiny{This document was prepared as an account of work sponsored by an agency of the United States government. Neither the United States government nor Lawrence Livermore National Security, LLC, nor any of their employees makes any warranty, expressed or implied, or assumes any legal liability or responsibility for the accuracy, completeness, or usefulness of any information, apparatus, product, or process disclosed, or represents that its use would not infringe privately owned rights. Reference herein to any specific commercial product, process, or service by trade name, trademark, manufacturer, or otherwise does not necessarily constitute or imply its endorsement, recommendation, or favoring by the United States government or Lawrence Livermore National Security, LLC. The views and opinions of authors expressed herein do not necessarily state or reflect those of the United States government or Lawrence Livermore National Security, LLC, and shall not be used for advertising or product endorsement purposes.}

\end{document}